\documentclass[a4paper,11pt]{article}
\usepackage{graphicx}
\usepackage{mathptmx}
\usepackage{amssymb,amsmath, amsfonts, amsthm}
\usepackage{a4wide}
\usepackage{color}

\usepackage{url}
\usepackage{tikz}
\usepackage{pgfplots}
\pgfplotsset{compat=1.14}
\usepackage{comment}
\usepackage{cases}
\usepackage{float}
\usepackage{array}
\usepackage{tabularx}
\usepackage{ragged2e}
\usepackage{comment}
\usepackage{multirow}
\usepackage{blkarray}
\usepackage{array}
\usepackage{arydshln}

\newcommand{\red}[1]{\textcolor{red}{#1}}

\newcommand{\email}[1]{{\tt #1}}
\newcommand{\R}{\mathbb{R}}

\newcommand{\norm}[1]{\|#1\|}
\newcommand{\bnorm}[1]{\big\|#1\big\|}

\newcommand{\dist}[1]{{\rm dist}(#1)}
\newcommand{\sign}{{\rm sgn}}

\newcommand{\mv}{\,\big |\,}

\newcommand{\B}{{\cal B}}

\newcommand{\I}{{\cal I}}

\newcommand{\Sp}{{\mathcal S}}

\newcommand{\Z}{{\cal Z}}

\newcommand{\setto}[1]{\mathop{\rightarrow}\limits^#1}
\newcommand{\longsetto}[1]{{\mathop{\longrightarrow}\limits^{#1}}}
\newcommand{\skalp}[1]{\langle #1\rangle}

\newcommand{\xb}{\bar x}
\newcommand{\yb}{\bar y}
\newcommand{\zb}{\bar z}

\newcommand{\AT}[2]{{\textstyle{#1\atop#2}}}

\newcommand{\oo}{o}
\newcommand{\OO}{{\cal O}}
\newcommand{\argmin}{\mathop{\rm arg\,min}}

\newcommand{\cl}{\mathrm{cl}\,}

\newcommand{\co}{\mathrm{conv}\,}
\newcommand{\gph}{\mathrm{gph}\,}

\newcommand{\dom}{\mathrm{dom}\,}

\newcommand{\clm}{\mathrm{clm}\,}
\newcommand{\tto}{\rightrightarrows}
\newcommand{\Limsup}{\mathop{{\rm Lim}\,{\rm sup}}}
\newcommand{\Liminf}{\mathop{{\rm Lim}\,{\rm inf}}}

\newcommand{\myvec}[1]{\begin{pmatrix}#1\end{pmatrix}}
\newcommand{\SCD}{SCD\ }

\newcommand{\scdreg}{\mathrm{scd\,reg}\,}
\newcommand{\reg}{\mathrm{reg}\,}
\newcommand{\bnd}{\mathrm{bnd}\,}

\newcommand{\ssstar}{semismooth$^{*}$ }
\newcommand{\ee}[2]{{#1}^{(#2)}}
\newcommand\Uad{U_{\rm ad}}
\newcommand{\oR}{(-\infty,\infty]}
\newcommand{\onabla}{\overline\nabla}
\newcommand{\todo}[1]{\red{\bf ToDO: #1}}

\newcommand{\rge}{{\rm rge\;}}

\newcommand{\cocl}{{\rm cocl}\,}
\newcommand{\lip}{{\rm lip\,}}
\newlength{\myAlgBox}
\newtheorem{theorem}{Theorem}[section]
\newtheorem{proposition}[theorem]{Proposition}
\newtheorem{remark}[theorem]{Remark}
\newtheorem{lemma}[theorem]{Lemma}
\newtheorem{corollary}[theorem]{Corollary}
\newtheorem{definition}[theorem]{Definition}
\newtheorem{example}[theorem]{Example}

\if{
\newsiamremark{example}{Example}
\newsiamremark{remark}{Remark}
\newsiamremark{assumption}{Assumption}

\headers{Semismoothness in nonsmooth analysis}{H. Gfrerer, and J. V. Outrata}
}\fi
\title{On the role of semismoothness in nonsmooth numerical analysis: Theory
\thanks{submitted to the editors DATE
{\bf Funding: }{The work of the second author was supported by the Grant Agency of the Czech Republic, Project 21-06569K.}
}}
\author{Helmut Gfrerer\thanks{Johann Radon Institute for Computational and Applied Mathematics (RICAM), A-4040 Linz, Austria and Institute of Information Theory and Automation, Czech Academy of
Sciences, 18208 Prague, Czech Republic;
\email{helmut.gfrerer@ricam.oeaw.ac.at}
}
\and   Ji\v{r}\'{i} V. Outrata\thanks{Institute of Information Theory and Automation, Czech Academy of
Sciences, 18208 Prague, Czech Republic;
\email{outrata@utia.cas.cz}
}}

\date{}

\begin{document}
\maketitle
\if{
\begin{abstract}For the numerical solution of nonsmooth problems, sometimes it is not necessary that an exact subgradient/generalized Jacobian is at our disposal, but that a certain semismoothness property is fulfilled. In this paper we consider not only   semismoothness of nonsmooth real- and  vector-valued mappings, but also its interplay with  the semismoothness$^*$ property for multifunctions. In particular, we are interested in the semismoothness of solution maps to parametric semismooth$^*$ inclusions. Our results are expressed in terms of suitable generalized derivatives of the set-valued part, i.e., by limiting coderivatives or by SC (subspace containing) derivatives. As a byproduct we identify a class of multifunctions having the remarkable property that they are strictly proto-differentiable almost everywhere (with respect to some Hausdorff measure) on their graph.
\end{abstract}
}\fi
\begin{abstract}For the numerical solution of nonsmooth problems, sometimes it is not necessary that an exact subgradient/generalized Jacobian is at our disposal, but it suffices that a semismooth derivative, i.e., a mapping satisfying a certain  semismoothness property, is available.  In this paper we consider not only   semismooth derivatives  of single-valued mappings, but also its interplay with  the semismoothness$^*$ property for multifunctions. In particular, we are interested in semismooth derivatives of solution maps to parametric semismooth$^*$ inclusions. Our results are expressed in terms of suitable generalized derivatives of the set-valued part, i.e., by limiting coderivatives or by SC (subspace containing) derivatives. Further we show that semismooth derivatives coincide a.e. with generalized Jacobians and state some consequences concerning strict proto-differentiability for semismooth$^*$ multifunctions.
\end{abstract}
\section{Introduction}
As stated in numerous works in the literature,  semismoothness plays an important role in solving various nonsmooth problems. Let us list three typical problem classes of this sort.

\begin{enumerate}
\item The first one provides the main motivation for this paper and concerns  optimization problems
of the form
\begin{align}\label{EqBasProbl}\min\theta(x)\quad\mbox{ subject to }\quad x\in \Uad,
\end{align}
where $\theta:\R^n\to\R$ is a locally Lipschitz function and $\Uad\subset\R^n$ is a closed set. For the numerical solution of these problems, the use of {\em bundle methods} is a popular approach. There exists a bunch of bundle methods in the literature, cf. \cite{Le, Ki, SZ}, and all these methods have in common that they require an
\begin{equation}\label{EqOracle}\begin{minipage}{10cm}\mbox{\em oracle}, which returns for arbitrary $x\in\Uad$ the function value $\theta(x)$ and one subgradient $g(x)\in\partial^c\theta(x)$,\end{minipage}\end{equation}
where $\partial^c\theta$ denotes the Clarke subdifferential. Then, under the additional assumption that
\begin{equation*}\mbox{$\theta$ is weakly semismooth (see Remark \ref{RemWeaklySS} below),}\end{equation*}
it can be shown that the bundle method either stops after finitely many steps at a C-stationary point $\xb$ or, if the produced sequence $x_k$ remains bounded, one of its accumulation points is C-stationary.

From a limited number of bundle methods, applicable to the numerical solution of \eqref{EqBasProbl} (e.g. \cite{Le, Ki, SZ}), we consider in this paper the Bundle-Trust region algorithm (BT algorithm) from Schramm and Zowe \cite{SZ}, described also in \cite{Zo89}.
It seems that Dempe and Bard \cite{DeBa01} were the first who observed that the BT algorithm  also works, when the oracle delivers only a so-called {\em pseudogradient}. Pseudogradients were introduced by Norkin \cite{No78,No80} and have the remarkable feature that they fulfill a certain semismoothness property. Although this  property does not necessarily imply that $\theta$ is weakly semismooth, one can show that a similar convergence result  holds as above: If the BT-algorithm produces a bounded sequence then at least one accumulation point is {\em pseudo-stationary}, i.e., $0$ is a pseudogradient of $\theta$ at this point. By checking the convergence proof for other bundle methods \cite{Ki, Le} one can show
that a similar assertion  holds: If the oracle delivers only a pseudogradient, one can still prove that the method works and returns pseudo-stationary points.

\item Semismoothness also plays an important role  in another field of numerical analysis, namely the solution of nonsmooth equations
\begin{equation*}
G(x)=0,
\end{equation*}
where $G:\R^n\to\R^n$ is a locally Lipschitz mapping. Whereas in the original semismooth Newton method,  introduced by Kummer \cite{Kum} and Qi and Sun \cite{QiSun93}, matrices from the Clarke generalized Jacobian were used, it was soon observed that the Clarke generalized Jacobian can be replaced by another mapping which fulfills a certain semismoothness property, closely related to the one imposed on pseudogradients. In what follows the elements of this mapping, both in the scalar- and vector-valued case, will be termed {\em semismooth derivatives}.  This is particularly important for carrying over the semismooth Newton method to the infinite-dimensional setting, where no generalized Jacobian exists. We refer the reader to \cite{Ulb11} for the semismooth Newton method in the infinite-dimensional setting. Note that, whereas the Clarke calculus often yields only inclusions, the calculus of semismooth derivatives yields again a semismooth derivative. Though the semismoothness property, imposed on the semismooth derivatives in the scalar- and vector-valued case, is essentially the same, there is a big difference: Whereas for the semismooth Newton method the semismoothness needs to hold only at the solution, when working with pseudogradients in the framework of a bundle method it needs to hold at each $x\in U_{ad}$.


\item Finally, for the numerical solution of inclusions
\begin{equation*}
  0\in F(x),
\end{equation*}
the so-called semismooth$^*$ Newton method was introduced in \cite{GfrOut21}, where $F:\R^n\tto\R^n$ is a set-valued mapping. Here, the semismooth$^*$ property was defined in terms of the  {\em limiting (Mordukhovich) coderivative} and was later simplified for the important subclass of SCD (subspace containing derivatives) mappings by using the notion of SC (subspace containing) derivatives in \cite{GfrOut22a}. The SC derivative can be considered as a generalization of the B-Jacobian (Bouligand differential) of single-valued mappings to multifunctions and acts as a kind of skeleton for the limiting coderivative. The semismooth$^*$ Newton method has been already successfully  applied to several difficult problems, see, e.g., \cite{GfrOutVal23,GfrMaOutVal23,Gfr24,GfrHubRam25}.
\end{enumerate}

We observe that semismoothness plays an important role in solving different types of nonsmooth problems: optimization problems, equations and inclusions. The respective notions of semismoothness are slightly different for the different problem types 
and it is not clear how they can combined for the numerical solution of problems assembled by different nonsmooth and set-valued objects. As an example consider an optimization problem of the form
\begin{align}
\nonumber  \min\,&\varphi(x,y)\\
\label{EqCompProbl}  \mbox{subject to }&0\in F(x,y),\\
\nonumber &x\in \Uad,
\end{align}
where $\varphi:\R^n\times\R^m\to\R$ is a locally Lipschitz function, $F:\R^n\times\R^m\tto\R^m$ is a set-valued mapping and $\Uad\subset\R^n$ is a closed set. Problems of such type arise, e.g., in bilevel programming and mathematical programs with equilibrium constraints (MPECs).

One possibility for tackling \eqref{EqCompProbl} is the so-called {\em implicit programming approach} (ImP), cf. \cite{LuPaRa97}, where one typically requires that for every $x$ the inclusion $0\in F(x,\cdot)$ has a unique solution denoted by $\sigma(x)$. Hence we can eliminate $y$ from the model above and arrive at the (reduced) problem
\begin{align}
 \label{EqReducedProbl} \min\theta(x):=\varphi(x,\sigma(x))\quad \mbox{ subject to }x\in \Uad.
\end{align}
Under the additional assumption that $\sigma$ is locally Lipschitz, the objective $\theta$ is locally Lipschitz as well and we can apply a bundle method for its numerical solution, provided a suitable oracle is available.
In most cases, however, one cannot guarantee that at all points a Clarke subgradient of $\theta$ can be computed. Nevertheless, if the original data $\varphi$ and $F$ are semismooth and semismooth$^*$, respectively, then it is possible to construct an oracle returning an element of a semismooth derivative of $\theta$. This oracle will then ensure the convergence of the used bundle method to a stationary point.

In this paper we will provide  the theoretical base for this approach. In particular, we will conduct a deep analysis of various available notions of the semismoothness property and find out some important mutual relationships. We will show  that a single-valued mapping possesses some  semismooth derivative on an open set $\Omega$ if and only if it is G-semismooth (semismooth in the sense of Gowda \cite{Gow04}) there. Further, this mapping is almost everywhere strictly differentiable  and the semismooth derivative coincides with the generalized Jacobian almost everywhere in $\Omega$. Moreover, its generalized Jacobian is contained in another semismooth derivative obtained from the original one by performing some closure operation and taking the convex hull of the values.
 For set-valued mappings we introduce a generalization of the \ssstar property (denoted  $\Gamma$-ss$^*$ property) and the so-called SCD-ss$^*$ property. These new notions enable us, among other things, to analyze the semismoothness of solution maps to inclusions appearing in \eqref{EqCompProbl} by examination of the underlying mapping $F$. Note that, up to now, only results on the semismoothness property of single-valued solution mappings given implicitly by single-valued equations are available, cf. \cite{Sun01,PaSuSu03,Gow04,MeSuZh05,Kr18}.

As already observed by Chen, Sun and Zhang in the very recent paper \cite{ChSuZh25}, there are some relationships between G-semismoothness and the \ssstar property. We extend these results by showing that a graphically Lipschitzian mapping (i.e., a multifunction whose graph coincides with the graph of a Lipschitz single-valued mapping after a change of coordinates) is SCD-ss$^*$ if and only if the transformed mapping is G-semismooth. As a byproduct we obtain that an SCD-ss$^*$ mapping is almost everywhere strictly proto-differentiable and the  SC  derivative is almost everywhere a singleton showing the easiness of its computation there.

Finally, on the basis of the above theoretical results, we will derive for problem
(\ref{EqCompProbl}) two implementable algorithmic schemes, enabling us the computation of a semismooth derivative of $\theta$ given by \eqref{EqReducedProbl} via either the limiting coderivative or the SC derivatives of $F$. The respective oracles will then, in cooperation with a suitable bundle method, generate stationary points even in some situations outside the classical area of applicability of ImP. This is illustrated by an academic example. In a subsequent paper we intend to work out this approach to several instances of (\ref{EqCompProbl}) having a practical relevance and to demonstrate its efficiency by numerical experiments.

Let us mention that for single-valued mappings there is a close relation between the semismoothness property and conservative derivatives, the latter being explored in a recent paper by Bolte and Pauwels \cite{BolPau21}. In fact, it was shown in \cite{DavDru22} that both notions are equivalent for semialgebraic maps. As semismooth derivatives also obey a chain rule like conservative Jacobians, the considerations about automatic differentiation from \cite{BolPau21} carry over to our setting.
Recently, a conservative Jacobian for solution maps given implicitly by equations was established  in \cite{BoLePauSi21}. In \cite{BoPauSi24}, a conservative Jacobian for the solution map of a parametric monotone inclusion was obtained by rewriting the inclusion as an equivalent fixed-point equation for the forward-backward map by means of the resolvent. Though the resolvents  are explicitly known for a lot of monotone mappings, it is not clear how  a conservative Jacobian of the resolvent  can be computed in the general case. In fact, the resolvent itself may be considered as the solution map of a shifted inclusion. On the contrary, our approach is based on the differentiation of the underlying set-valued mapping by SC  derivatives.

The paper is organized in the following way. Section 2 is divided into several subsections with the first two  containing the basic notions of variational analysis and generalized differentiation. In Subsection \ref{SubSecClMap} we describe a general principle, analogous to the construction of Clarke generalized Jacobian, how to generate from an arbitrary multifunction another multifunction having some favorable properties. In the next subsection we recall the basic properties of SCD mappings. In Subsection \ref{SubSecSS} we recall the semismooth$^*$ property of set-valued mappings and its extension to SCD mappings.

The main new results are presented in Sections 3--5. Section 3 is devoted to the semismoothness property of single-valued mappings with respect to certain multifunctions. This concept is commonly used in semismooth Newton methods for solving nonsmooth equations,  in particular in infinite dimensions.
However, in contrast to these applications, we do not consider the case that semismoothness only holds at the solution, but on the whole domain.
There is a close relation with the seminal work of Norkin \cite{No78,No80} on generalized differentiation of real-valued functions. In Section 4 we present important relationships between the  G-semismooth and the SCD-ss$^*$ property for graphically Lipschitzian multifunctions. Section 5 concerns the semismooth$^*$ property of solution maps to inclusions. We establish the semismooth$^*$ property of general solution maps and the semismooth property for continuous localizations and selections, respectively. Finally it will be shown that the analysis considerably simplifies for SCD mappings.

The following notation is employed. Given a linear subspace $L\subseteq \R^n$, $ L^\perp$ denotes its orthogonal complement.
Further, given a multifunction $F$, $\gph F:=\{(x,y)\mv y\in F(x)\}$ stands for its
graph and $\dom F:=\{x\mv F(x) \ne \emptyset\}$ denotes its domain. For an element $u\in\R^n$, $\norm{u}$ denotes its Euclidean norm and  $\B_\delta(u)$ denotes the closed ball around $u$ with radius $\delta$. The closed unit ball is denoted by $\B$. In a product space we use the norm $\norm{(u,v)}:=\sqrt{\norm{u}^2+\norm{v}^2}$. The space of all $m\times n$ matrices is denoted by $\R^{m\times n}$. For a matrix $A\in\R^{m\times n}$ , we employ the operator norm $\norm{A}$ with respect to the Euclidean norm and we denote the range of $A$ by $\rge A$. Given a set $\Omega\subset\R^s$, we define the distance of a point $x$ to $\Omega$ by $\dist{x,\Omega}:=\inf\{\norm{y-x}\mv y\in\Omega\}$ and the indicator function is denoted by $\delta_\Omega$.
Further we define for a mapping $F:\Omega\to\R^m$, $\Omega\subset\R^n$ open, and $x\in\Omega$ the moduli
\[\lip F(x):=\limsup_{\AT{u,v\setto{\Omega} x}{u\not=v}}\frac{\norm{F(u)-F(v)}}{\norm{u-v}},\quad \clm F(x):=\limsup_{u\setto{\Omega}x}\frac{\norm{F(u)-F(x)}}{\norm{u-x}}.\]
Clearly, $F$ is Lipschitz near $x$ (calm at $x$) if and only if $\lip F(x)$ ($\clm F(x)$) is finite.


\section{Background from variational analysis and preliminary results}

\subsection{Nonsmooth analysis}
Given a parameterized family $C_t$ of subsets of a metric space, where $t$ also belongs to a metric space, we define the upper (outer) and lower(inner) limits in the sense of Painlev\'e--Kuratowski as
\begin{align*}
\Limsup_{t\to\bar t}C_t:=\{x\mv \liminf_{t\to\bar t}\,\dist{x,C_t}=0\},\
\liminf_{t\to\bar t}C_t:=\{x\mv \limsup_{t\to\bar t}\,\dist{x,C_t}=0\}.
\end{align*}

\if{
Let us recall the notions of upper (outer) and lower (inner) limits in the sense of Painlev\'e--Kuratowski of a parameterized family $C_t$ of subsets of a metric space, where $t$ also belongs to a metric space:
\begin{align*}
\Limsup_{t\to\bar t}C_t:=\{x\mv \liminf_{t\to\bar t}\,\dist{x,C_t}=0\},\\
\Liminf_{t\to\bar t}C_t:=\{x\mv \limsup_{t\to\bar t}\,\dist{x,C_t}=0\}.
\end{align*}
It easily follows from the definition that both the upper and lower limit sets are
closed. These sets can also be described in terms of sequences as follows. The
upper limit can be defined as the set of points $x$ for which there exists
a sequence $t_k\to\bar t$ such that $x_k\to x$ for some $x_k\in C_{t_k}$ $\forall k$. Similarly, the lower
limit  can be defined as a set of points $x$ such that for every sequence
$t_k\to\bar t$ we can find $x_k\in C_{t_k}$ such that $x_k\to x$. Clearly, $\Liminf_{t\to\bar t}C_t\subset \Limsup_{t\to\bar t}C_t$.
}\fi

Consider a mapping $F:\Omega\to \R^m$, where $\Omega\subset\R^n$ is an open set. In the sequel we denote by $\OO_F$ the set of points where $F$ is (Fr\'echet) differentiable and we denote the Jacobian by $\nabla F(x)$, $x\in\OO_F$. We use this notation also for scalar-valued functions $f:\R^n\to\R$, i.e., $\nabla f(x)$ is an $1\times n$ matrix, i.e., a row vector.

  The {\em B-Jacobian (B-differential)} is defined as
\[\overline \nabla F(\xb)= \Limsup_{x\longsetto{\OO_F}\xb}\{\nabla F(x)\},\ \xb\in \Omega.\]
Then the Clarke generalized Jacobian is given by $\co \onabla F(\xb)$.

If $F$ is locally Lipschitz, by Rademacher's Theorem the set $\Omega\setminus\OO_F$  has Lebesgue measure $0$.
Since the Jacobians $\nabla F(x)$, $x\in\OO_F$, are also locally bounded
by the Lipschitz constant, we obtain $\emptyset\not=\overline \nabla F(x)\subset\co\overline \nabla F(x)$, $x\in \Omega$.
Further, we say that $F$ is {\em semismooth} at $\xb\in\Omega$ if $F$ is directionally differentiable at $\xb$ and
\begin{equation*}
  \lim_{\AT{x\setto{\Omega}\xb}{x\not=\xb}}\frac{\sup\{\norm{F(x)-F(\xb)-A(x-\xb)}\mv A\in \co\onabla F(x)\}}{\norm{x-\xb}}=0.
\end{equation*}
Note that this definition of semismoothness is different but equivalent to the original ones by Mifflin \cite{Mif77} (for the scalar case) and by Qi and Sun \cite{QiSun93}. If the requirement of directional differentiability is dropped we arrive at the semismooth property as defined by Gowda \cite{Gow04}, which is also called {\em G-semismoothness} in the literature.

\if{
\todo{shorten}
}\fi
We say that $F$ is {\em strictly differentiable} at $\xb\in \Omega$ if  $F$ is Lipschitz near $\xb$ and there is an $m\times n$ matrix $A$ such that
\[\lim_{\AT{x\to\xb}{t\downarrow 0}}\frac{F(x+tu)-F(x)}t=Au,\ u\in\R^n.\]

For a locally Lipschitz function $f:\Omega\to\R$, $\Omega\subset\R^n$ open, we can define the {\em Clarke subdifferential} as
\[\partial^c f(x) :=\co \overline \nabla f(x)^T:=\{g^T\mv g\in \co \overline \nabla f(x)\},\ x\in \Omega.\]
The elements of $\partial^c f(x)$ are called {\em subgradients} and are column vectors.

\if{
Given any extended real-valued function $f:\R^n\to\oR$, we always assume that $f$ is proper, i.e., $f(x)>-\infty$ for all $x$ and $\dom  f:=\{x\in\R^n\mv  f(x)<\infty\}\not=\emptyset$. Given a point $\xb\in \dom  f$, the {\em regular (Fr\'echet) subdifferential} of $f$ at $\xb$ is given by
\[\widehat\partial  f(\xb):=\Big\{x^*\in\R^n\mv\liminf_{x\to\xb}\frac{ f(x)- f(\xb)-\skalp{x^*,x-\xb}}{\norm{x-\xb}}\geq 0\Big\},\]
while the {\em (limiting) subdifferential} is defined by
\[\partial  f(\xb):=\{x^*\mv \exists (x_k,x_k^*)\longsetto{\gph \widehat \partial f} (\xb, x^*) \mbox{ with } f(x_k)\to f(\xb)\}.\]
If $\dom f=:\Omega$ is open, then clearly $\OO_f\subset \dom\widehat\partial f$ and $\onabla f(\xb)^T\subset\partial f(\xb)$, provided that $f$ is continuous at $\xb\in \Omega$. If $f$ is locally Lipschitz on $\Omega$ then
\[\partial^c f(x)=\co \partial f(x),\ x\in \Omega,\]
cf. \cite[Theorem 9.61]{RoWe98}.
}\fi
\begin{lemma}\label{LemStrictDiff}
  Let $\Omega\subset\R^n$ be open and let $F:\Omega\to\R^m$  be  Lipschitz near $\xb\in \Omega$. Then the following statements are equivalent:
  \begin{enumerate}
    \item[(i)] $F$ is strictly differentiable at $\xb$.
    \item[(ii)] $\onabla F(\xb)$ is a singleton.
    \item[(iii)] $\lim_{x\longsetto{\OO_F}\xb}\nabla F(x)$ exists.
  \end{enumerate}
\end{lemma}
\begin{proof}
  ad (i) $\Rightarrow$ (iii): If $F$ is strictly differentiable at $\xb$ then $\nabla F$ is continuous at $\xb$ relative to $\OO_F$, cf. \cite[Exercise 9.64]{RoWe98}. Thus $\lim_{x\longsetto{\OO_F}\xb}\nabla F(x)$ exists (and equals to $\nabla F(\xb)$).\par
  ad (iii) $\Rightarrow$ (ii): This is evident from the definitions.\par
  ad (ii) $\Rightarrow$ (i): Let $\onabla F(\xb)=\{A\}$. For every $i=1,\ldots,m$, we conclude from \cite[Theorem 2.6.6]{Cla83} that the $i$-th component of $F$, $i=1,\ldots m$, satisfies $\bar \nabla F_i(\xb)=\{A_i\}$, where $A_i$ denotes the $i$-th row of $A$. Thus $\partial^cF_i(\xb)$ is a singleton and therefore $F_i$ is strictly differentiable at $\xb$ by \cite[Proposition 2.2.4]{Cla83}. Since every component $F_i$, $i=1,\ldots m$, is strictly differentiable at $\xb$, so is $F$ as well.
\end{proof}

\subsection{Variational geometry and differentiation of multifunctions}
 \begin{definition}
 Let $\Omega\subset\R^{s}$ and $\bar{x} \in \Omega$.
\begin{enumerate}
 \item [(i)]The  {\em tangent (contingent, Bouligand) cone} $T_{\Omega}(\bar{x})$, the {\em paratingent cone} $T^P_\Omega(\xb)$ and the {\em regular (Clarke) tangent cone} $\widehat T_\Omega(\xb)$ to $\Omega$ at $\bar{x}$ are given by
 \[T_{\Omega}(\bar{x}):=\Limsup\limits_{t\downarrow 0} \frac{\Omega-\bar{x}}{t},\quad
 T^P_\Omega(\xb):=\Limsup\limits_{\AT{x\setto{{\Omega}}\xb}{t\downarrow 0}} \frac{\Omega- x}{t},\quad
 \widehat T_\Omega(\xb):=\Liminf_{\AT{x\setto{\Omega}\xb}{t\downarrow 0}}\frac{\Omega-x}t.\]
 \if{
 \item [(i)]The  {\em tangent (contingent, Bouligand) cone} to $\Omega$ at $\bar{x}$ is given by
 \[T_{\Omega}(\bar{x}):=\Limsup\limits_{t\downarrow 0} \frac{\Omega-\bar{x}}{t}\]
   and the {\em paratingent cone} to $\Omega$ at $\xb$ is given by
 \[T^P_\Omega(\xb):=\Limsup\limits_{\AT{x\setto{{\Omega}}\xb}{t\downarrow 0}} \frac{\Omega- x}{t}.\]
 The {\em regular (Clarke) tangent cone} to $\Omega$ at $\xb$ amounts to
 \[\widehat T_\Omega(\xb):=\Liminf_{\AT{x\setto{\Omega}\xb}{t\downarrow 0}}\frac{\Omega-x}t.\]
 }\fi
 \item[(ii)] The set $\widehat{N}_{\Omega}(\bar{x}):=(T_{\Omega}(\bar{x}))^{\circ}$,
 i.e. the (negative) polar cone to the tangent cone $T_\Omega(\xb)$, is the {\em regular (Fr\'{e}chet) normal cone} to $\Omega$ at $\bar{x}$, and
 \[N_{\Omega}(\bar{x}):=\Limsup\limits_{x\setto{\Omega}\xb} \widehat{N}_{\Omega}(x)\]
 is the {\em limiting (Mordukhovich) normal cone} to $\Omega$ at $\bar{x}$.
 \end{enumerate}
\end{definition}
If $\Omega$ is convex, then $\widehat{N}_{\Omega}(\bar{x})= N_{\Omega}(\bar{x})$ amounts to the classical normal cone in the sense of convex analysis and we will  write $N_{\Omega}(\bar{x})$.

The above listed cones enable us to describe the local behavior of set-valued maps via various
generalized derivatives.

\begin{definition}
Consider a  multifunction $F:\R^n\tto\R^m$ and let $(\xb,\yb)\in \gph F$.
\begin{enumerate}
\item[(i)]
 The multifunction $DF(\xb,\yb):\R^n\tto\R^m$ given by $\gph DF(\xb,\yb)=T_{\gph F}(\xb,\yb)$ is called the {\em graphical derivative} of $F$ at $(\xb,\yb)$.
\item[(ii)]
 The multifunction $D_*F(\xb,\yb):\R^n\tto\R^m$ given by $\gph D_*F(\xb,\yb)=T^P_{\gph F}(\xb,\yb)$ is called the {\em  strict (paratingent) derivative} of $F$ at $(\xb,\yb)$.
\item[(iii)]
 The multifunction $\widehat D^\ast F(\xb,\yb ): \R^m\tto\R^n$  defined by
 \[ \gph \widehat D^\ast F(\xb,\yb )=\{(y^*,x^*)\mv (x^*,-y^*)\in \widehat N_{\gph F}(\xb,\yb)\}\]
is called the {\em regular (Fr\'echet) coderivative} of $F$ at $(\xb,\yb )$.
\item [(iv)]  The multifunction $D^\ast F(\xb,\yb ): \R^m \tto \R^n$,  defined by
 \[ \gph D^\ast F(\xb,\yb )=\{(y^*,x^*)\mv (x^*,-y^*)\in N_{\gph F}(\xb,\yb)\}\]
is called the {\em limiting (Mordukhovich) coderivative} of $F$ at $(\xb,\yb )$.
\end{enumerate}
\end{definition}

 We call $F$ {\em single-valued}, if for every $x\in \dom F$ the set $F(x)$ is a singleton. In this case, we can omit the second argument and write $DF(x)$, $\widehat D^*F(x),\ldots$ instead of $DF(x,F(x))$, $\widehat D^*F(x,F(x)),\ldots$.
\if{Given a mapping $F:\Omega\subset\R^n\to\R^m$, we can identify it with a set-valued single-valued mapping $\R^n\tto\R^m$ via
\[x\to\begin{cases}\{F(x)\}&\mbox{if $x\in \Omega$}\\\emptyset&\mbox{else.}\end{cases}\]
}\fi
\if{
However, be aware that when considering limiting objects at $x$ where $F$ is not continuous, it is not enough to consider only sequences $x_k\to x$ but we must work with sequences $(x_k,F(x_k))\to(x,F(x))$.
}\fi

If a mapping $F:\Omega\to\R^m$, $\Omega\subset\R^n$ open, is differentiable at $\xb\in \Omega$ then
$DF(\xb)(u)=\nabla F(\xb)u,$ $u\in\R^n.$ If $F$ is even strictly differentiable at $\xb$, then also
\begin{equation}\label{EqStrictDeriv_StrictDiff}D_*F(\xb)(u)=\nabla F(\xb)u,\ u\in\R^n,\quad D^*F(\xb)(v^*)=\nabla F(\xb)^Tv^*,\ v^*\in\R^m\end{equation}
holds, cf. \cite{RoWe98}. A generalization of strict differentiability  to general set-valued mappings is provided by the notion of strict proto-differentiability.
\if{
Following \cite{Ro89}, a set-valued mapping $F:\R^n\tto\R^m$ is called {\em proto-differentiable} at $(\xb,\yb)\in\gph F$ if the set-valued mappings
\begin{equation*} \Delta_{\xb,\yb,t}F:\xi\mapsto \big(F(\xb+t\xi)-\yb)/t\quad \mbox{ for $t> 0$}\end{equation*}
graph-converge as $t \downarrow 0$, i.e., their graphs converge as subsets of $\R^n\times\R^m$. Of course, the limit must be equal to the graphical derivative $DF(\xb,\yb)$.
This is the same as requiring that
\[\Limsup_{t\downarrow 0}\frac{\gph F-(\xb,\yb)}t= \Liminf_{t\downarrow 0}\frac{\gph F-(\xb,\yb)}t\big(=\gph DF(\xb,\yb)\big).\]
}\fi
According to \cite{PolRo96} we say that a set-valued mapping $F:\R^n\tto\R^m$ is {\em strictly proto-differentiable} at $(\xb,\yb)$ if the mappings
\begin{equation*} \Delta_{x,y,t}F:\xi\mapsto \big(F(x+t\xi)-y)/t\quad \mbox{ for $t> 0$}\end{equation*}
graph-converge as $t\downarrow 0$ and $(x,y)\longsetto{\gph F}(\xb,\yb)$, the limit being $D_*F(\xb,\yb)$. This means that
\begin{equation}\label{EqStrictProto}\Limsup_{\AT{(x,y)\longsetto{\gph F}(\xb,\yb)}{t\downarrow 0}}\frac{\gph F-(x,y)}t= \Liminf_{\AT{(x,y)\longsetto{\gph F}(\xb,\yb)}{t\downarrow 0}}\frac{\gph F-(x,y)}t\big(=\gph D_*F(\xb,\yb)\big).\end{equation}
\begin{lemma}\label{LemStrictProto}
  If $F:\R^n\tto\R^m$ is strictly proto-differentiable at $(\xb,\yb)\in\gph F$  and $\gph F$ is locally closed at $(\xb,\yb)$ then $T_{\gph F}(\xb,\yb)$ is a subspace of $\R^n\times\R^m$ and
  \begin{align}\label{EqTangStrictProto}
    \gph DF(\xb,\yb) &= T_{\gph F}(\xb,\yb)=\Limsup_{(x,y)\longsetto{\gph F}(\xb,\yb)}T_{\gph F}(x,y)=\Liminf_{(x,y)\longsetto{\gph F}(\xb,\yb)}T_{\gph F}(x,y)\\
    \nonumber&=\widehat T_{\gph F}(\xb,\yb)=\gph D_*F(\xb,\yb).
  \end{align}
\end{lemma}
\begin{proof}
  By taking into account \cite[Theorem 6.26]{RoWe98}, we obtain from \eqref{EqStrictProto} that
  \begin{align*}
  \gph DF&(\xb,\yb)=T_{\gph F}(\xb,\yb)\subset \gph D_*F(\xb,\yb)=\Limsup_{{(x,y)\longsetto{\gph F}(\xb,\yb)},\ { t\downarrow 0}}\frac{\gph F-(x,y)}t\\
  &= \Liminf_{{(x,y)\longsetto{\gph F}(\xb,\yb)},\ {t\downarrow 0}}\frac{\gph F-(x,y)}t=\Liminf_{(x,y)\longsetto{\gph F}(\xb,\yb)}T_{\gph F}(x,y)=\widehat T_{\gph F}(\xb,\yb)\subset T_{\gph F}(\xb,\yb).
  \end{align*}
  Hence, in order to prove \eqref{EqTangStrictProto} there remains to show the equality
  \[A:=\Limsup_{(x,y)\longsetto{\gph F}(\xb,\yb)}T_{\gph F}(x,y)=\Limsup_{(x,y)\longsetto{\gph F}(\xb,\yb),\ t\downarrow 0}\big(\gph F-(x,y)\big)/t=:B.\]
  Clearly, by the inclusions above we have $B=\gph D_*F(\xb,\yb)=T_\gph F(\xb,\yb)\subset A$.
  \if{
  Now consider $(u,v)\in A$ together with sequences $(x_k,y_k)\longsetto{\gph F}(\xb,\yb)$ and $(u_k,v_k)\in T_{\gph F}(x_k,y_k)$ satisfying $(u_k,v_k)\to(u,v)$. \if{By definition of the tangent cone,}\fi For every $k$ we can find some $t_k\in (0,\frac 1k)$ and some $(\hat x_k,\hat y_k)\in\gph F$ with
  \[\norm{\frac{(\hat x_k,\hat y_k)-(x_k,y_k)}{t_k}-(u_k,v_k)}\leq\frac 1k\]
  implying $\big((\hat x_k,\hat y_k)-(x_k,y_k)\big)/t_k\to (u,v)$. This proves $(u,v)\in B$ and the equality $A=B$ is established.
}\fi
By \cite[Proposition 2.1(b)]{HaSa23} the reverse inclusion $A\subset B$ always holds and the equality $A=B$ is established.

  Finally, in order to show that $T_{\gph F}(\xb,\yb)$ is a subspace, observe first that $T_{\gph F}(\xb,\yb)=\gph D_*F(\xb,\yb)=\widehat T_{\gph F}(\xb,\yb)$ is a convex cone. By \cite[Proposition 2.1(a)]{HaSa23},  the strict derivative also satisfies $\gph D_*F(\xb,\yb)=-\gph D_*F(\xb,\yb)$ and thus  $T_{\gph F}(\xb,\yb)$  is a subspace.
\end{proof}
\if{
In case of single-valued Lipschitz maps, strict proto-differentiability coincides with strict differentiability.

\begin{lemma}
  Let $\Omega\subset\R^n$ be open and let $F:\Omega\to\R^m$  be  Lipschitz near $\xb\in \Omega$. Then $F$ is strictly proto-differentiable at $(\xb,F(\xb))$ if and only if $F$ is strictly differentiable at $\xb$.
\end{lemma}
\begin{proof}
By Lipschitz continuity of $F$, for all  convergent sequences $t_k\downarrow 0$, $u_k\to u$ and $x_k\to \xb$ the sequence $(F(x_k+t_ku_k)-F(x_k))/t_k$ has at least one accumulation point. Thus, strict proto-differentiability of $F$ at $\xb$ is equivalent with the requirement that for such sequences $t_k$, $u_k$, $x_k$ the limits
\[\lim_{k\to\infty}\frac{F(x_k+t_ku_k)-F(x_k)}{t_k}=\lim_{k\to\infty}\frac{F(x_k+t_ku)-F(x_k)}{t_k}\]
exist and are equal to $D_*F(\xb)(u)$. Hence, if $F$ is strictly proto-differentiable at $(\xb,F(\xb))$ then $D_*F(\xb)$   is single-valued, $\dom D_*F(\xb)=\R^n$ and $\gph D_*F(\xb)$ is a subspace by Lemma  \ref{LemStrictProto}. We conclude that $D_*F(\xb,\yb)$ is in fact a linear mapping from $\R^n$ to $\R^m$ and therefore there exists an $m\times n$ matrix $A$ with $D_*F(\xb)(u)=Au$, $u\in\R^n$. Thus $F$ is strictly differentiable at $\xb$. The reverse implication easily follows from the definition of strict differentiability and \eqref{EqStrictDeriv_StrictDiff}.
\end{proof}
}\fi

\subsection{Closure of mappings\label{SubSecClMap}}
It is well-known that the Clarke subdifferential $\partial^c\theta$ of a Lipschitz continuous function $\theta:\Omega\to\R$, $\Omega\subset\R^n$ open, is upper semicontinuous (usc) and has nonempty, convex and compact values $\partial^c\theta(x)$ for every $x\in\Omega$.
\if{
has the following properties:
\begin{itemize}
  \item $\partial^c \theta$ is locally bounded
  \item The set $\partial^c \theta(x)$ is nonempty, convex and compact for every $x\in \Omega$.
  \item $\partial^c \theta$ is upper semicontinuous on $\Omega$.
\end{itemize}
}\fi
We now discuss a possibility of how to extend an arbitrary locally bounded mapping to a mapping having these properties.

Given a mapping $\Psi:X\tto Y$ between two  metric spaces $X$ and $Y$, we denote by $\cl \Psi$ the mapping whose graph equals to the closure of $\gph\Psi$,  i.e.,
\[\gph(\cl\Psi) =\cl (\gph\Psi)= \{(x,y)\mv (x,y)=\lim_{k\to\infty}(x_k,y_k)\mbox{ for some sequence }(x_k,y_k)\in\gph \Psi\}.\]
Obviously, $\cl(\cl\Psi)=\cl\Psi$ and for every $x\in \dom(\cl\Psi)$ we have that $(\cl \Psi)(x)$ is closed. Further, $\dom(\cl \Psi)\subset\cl(\dom\Psi)$ and equality holds, e.g., when the metric space $Y$ is compact.

In most cases we consider set-valued mappings between Euclidean spaces.
The mapping $\Psi:\R^n\tto\R^m$ is called {\em locally bounded}
at $\xb\in \R^n$ if there is a neighborhood $U$ of $\xb$ and a nonnegative real $M$ such that
    \begin{equation}\label{EqBndDef}\norm{y}\leq M\ \forall y\in \Psi(x)\ \forall x\in U.\end{equation}
We denote by $\bnd \Psi(\xb)$ the infimum of all nonnegative reals $M$ such that \eqref{EqBndDef} holds for some neighborhood $U$ with the convention  $\bnd\Psi(\xb):=\infty$ if $\Psi$ is not locally bounded at $\xb$.
    \if{
    By the definition, the function $x\mapsto \bnd\Psi(x)$ is usc.
    }\fi
We claim that the function $\bnd\Psi:\R^n\to\R\cup\{\infty\}$ is usc, i.e., $\limsup_{x\to\xb}\bnd\Psi(x)\leq\bnd\Psi(\xb)$, $\xb\in\R^n$. Indeed, fix any $\xb \in\R^n$. If $\bnd\Psi(\xb)=\infty$ there is nothing to prove. If  $\bnd\Psi(\xb)<\infty$ for every $\epsilon>0$ we can find an open neighborhood $U$ of $\xb$ such that \eqref{EqBndDef} holds with $M=\bnd\Psi(\xb)+\epsilon$. Then it immediately follows that $\bnd\Psi(x)\leq\bnd\Psi(\xb)+\epsilon$, $x\in U$, verifying upper semicontinuity of $\bnd\Psi$ at $\xb$ because $\epsilon>0$ can be taken arbitrarily small.

Further, we denote by $\cocl \Psi$ the mapping defined by
\[(\cocl \Psi)(x):=\co\big((\cl\Psi)(x)\big),\ x\in \R^n.\]
It is easy to see that $\dom(\cocl\Psi)=\dom(\cl\Psi)$. Moreover, if $\Psi$ is locally bounded at $x\in\R^n$ then  $\cl\Psi$ and $\cocl \Psi$ are also locally bounded at $x$ and $(\cocl\Psi)(x)$ is a convex and compact set. Further, both $\cl\Psi$ and $\cocl \Psi$ are usc at $x$.

\if{
\begin{definition}
  Consider a mapping $\Psi:\R^n\tto\R^m$.
  \begin{enumerate}
    \item $\Psi$ is called {\em locally bounded} at $\xb\in \R^n$ if there is a neighborhood $U$ of $\xb$ and a nonnegative real $M$ such that
    \begin{equation}\label{EqBndDef}\norm{y}\leq M\ \forall y\in \Psi(x)\ \forall x\in U.\end{equation}
    We denote by $\bnd \Psi(\xb)$ the infimum of all reals $M$ such that \eqref{EqBndDef} holds for some neighborhood $U$.
    \item $\Psi$ is called {\em outer semicontinuous (osc)} at $\xb\in\R^n$ if $(\cl \Psi)(\xb)=\Psi(\xb)$.
    \item $\Psi$ is called {\em upper semicontinuous (usc)} at $\xb$ if $\Psi(\xb)$ is closed and for every open set  $V\supset\Psi(\xb)$ the set $\{x\mv \Psi(x)\subset V\}$ is a neighborhood of $\xb$.
  \end{enumerate}
  These terms are invoked {\em on $X$}, a subset of $\R^n$, when the property holds for every $x\in X$. The set $X$ will not be mentioned if $X=\R^n$.
\end{definition}
By the definition, the function $x\mapsto \bnd\Psi(x)$ is usc.

Given $\Psi:\R^n\tto\R^m$, we denote by $\cocl \Psi$ the mapping defined by
\[(\cocl \Psi)(x):=\co\big((\cl\Psi)(x)\big),\ x\in \R^n.\]
It is easy to see that $\dom(\cocl\Psi)=\dom(\cl\Psi)$.
\begin{lemma}\label{LemCocl}
  Let $\Psi:\R^n\tto\R^m$ be a  mapping. If $\Psi$ is locally bounded at $x\in\R^n$ then   $\cl\Psi$ and $\cocl \Psi$ are also locally bounded at $x$ and $(\cocl\Psi)(x)$ is a convex and compact set. Further, both $\cl\Psi$ and $\cocl \Psi$ are both osc and usc at $x$.
\end{lemma}
\begin{proof}
   Let $\Psi$ be locally bounded at $x\in\R^n$ and consider an open neighborhood $U$ of $x$ such that $V:=\bigcup_{x'\in U}\Psi(x')$ is bounded. We claim that
   \[\bigcup_{x'\in U}(\cl{\Psi})(x')\subset\bigcup_{x'\in U}(\cocl{\Psi})(x')\subset\co(\cl V).\]
    Indeed, consider a point $(x',y')\in\gph (\cl\Psi)\cap(U\times\R^m)$ together with a sequence $(x_k,y_k)\longsetto{\gph \Psi}(x',y')$. Then for all $k$ sufficiently large we have $x_k\in U$ and consequently $y_k\in V$. Thus $y'\in\cl V$ and $\bigcup_{x'\in U}(\cocl{\Psi})(x')\subset\co(\cl V)$ follows. This verifies the local boundedness of $\cocl\Psi$ at $x$. Since the set $(\cl \Psi)(x)$ is closed and bounded, it is compact and the same holds true for its convex hull $(\cocl \Psi)(x)$.

    Due to the local boundedness of $\cl\Psi$  and $\cocl\Psi$ at $x$, the mappings $\cl\Psi$ and $\cocl\Psi$ are osc if and only if they are usc, cf. \cite[Theorem 5.19]{RoWe98}. Clearly, $\cl\Psi$ is osc at $x$ and there remains to show that $\cocl\Psi$ is usc at $x$. If $x\not\in \dom(\cocl\Psi)=\dom(\cl\Psi)$ then the set  $\{x'\mv x'\not\in\dom(\cl\Psi)=\dom(\cocl\Psi)\}$ is a neighborhood of $x$ since $\cl\Psi$ is usc at $x$.  Now let $x\in\dom(\cocl\Psi)$ and consider an open set $V\supset(\cocl{\Psi})(x)$. Since $(\cocl{\Psi})(x)$ is compact, we can find some $\epsilon>0$ such that the convex open set $(\cocl{\Psi})(x)+\epsilon\B$ is contained in $V$. Since the mapping $\cl \Psi$ is usc, the set $U:=\{x'\mv (\cl\Psi)(x')\subset (\cl\Psi)(x)+\epsilon\B\}$ is a neighborhood of $x$. For every $x'\in U$ we have \[(\cocl{\Psi})(x')=\co\big((\cl\Psi)(x')\big)\subset\co\big((\cl\Psi)(x)+\epsilon\B\big)=(\cocl{\Psi})(x)+\epsilon \B\subset V\]
     and the lemma is proved.
\end{proof}
}\fi
\if{
\begin{example}
  Let $F:\Omega\to\R^m$, $\Omega\subset\R^n$ open, be locally Lipschitz and consider the mapping $\Psi_F:\R^n\tto\R^{m\times n}$ given by
  $\Psi_F(x):=\{\nabla F(x)\},\ x\in \dom \Psi_F:=\OO_F$.
  Then for every $x\in\Omega$ there holds $(\cl \Psi_F)(x)=\overline \nabla F(x)$ and $(\cocl \Psi_F)(x)=\co\overline \nabla F(x)$.
\end{example}
}\fi

\subsection{SCD mappings}
In this subsection we collect only some basic facts from the theory of SCD-mappings which are needed in this paper. We refer the interested reader to \cite{GfrOut22a,GfrOut23}  for more details and more far-reaching properties.

Let us denote by $\Z_{n,m}$ the metric space of all $n$-dimensional subspaces of $\mathbb{R}^{n+m}$ equipped with the metric
\begin{equation*}
d_{\Z_{n,m}}(L_{1},L_{2}):=\|P_{1}-P_{2}\|,
\end{equation*}
where $P_{i}$ is the symmetric $(n+m) \times (n+m)$ matrix representing the orthogonal projection onto $L_{i}, i=1,2$. By \cite[Lemma 3.1]{GfrOut23}, $\Z_{n,m}$ is (sequentially) compact.

For notational reasons, we usually identify $\R^{n+m}$ with $\R^n\times\R^m$, i.e., for a subspace $L\in \Z_{n,m}$ we  denote its elements by  $(u,v)$ instead of $\myvec{u\\v}$. Analogously, given two matrices $A\in\R^{n\times n}$, $B\in\R^{m\times n}$ so that $L=\rge\begin{pmatrix}A\\B\end{pmatrix}$, we write $L=\rge(A,B)$.
\if{
\begin{lemma}[cf. {\cite[Lemma 3.1]{GfrOut23}}]\label{LemBasic}
 The metric space $\Z_{n,m}$ is (sequentially) compact.
\end{lemma}
}\fi


 With  each $L \in \Z_{n,m}$ one can associate its {\em adjoint} subspace $L^{*}$ defined by
  \begin{equation*}
L^{*}:=\{(-v^{*},u^{*}) \in \mathbb{R}^{m} \times \R^{n} \mv (u^{*},v^{*})\in L^{\perp}\}.
  \end{equation*}
  Since $\dim L^{\perp}=m$, it follows that $L^{*}\in \Z_{m,n}$ (i.e., its dimension is $m$).
  \if{
It is easy to see that \todo{shorten?}
\begin{equation*}
L^{*}= S_{nm} L^{\perp}, \mbox{ where } S_{nm} =
\left( \begin{array}{lc}
0 & -I_{m}\\
I_{n} & 0
\end{array}\right).
\end{equation*}
}\fi
Further we have $(L^*)^*=L$ and $d_{\Z_{m,n}}(L_1^*,L_2^*)=d_{\Z_{n,m}}(L_1,L_2)$ for all subspaces $L_1,L_2\in\Z_{n,m}$ and thus the mapping $L\to L^*$ is an isometry between $\Z_{n,m}$ and $\Z_{m,n}$, cf. \cite{GfrOut23}.
In what follows, the symbol $L^*$ signifies both the adjoint subspace to some $L\in \Z_{n,m}$ as well as an arbitrary subspace from $\Z_{m,n}$. This double role, however, cannot lead to a confusion.

Consider now a mapping $F: \mathbb{R}^{n}  \rightrightarrows \mathbb{R}^{m}$.

\begin{definition}
We say that $F$ is {\em graphically smooth of dimension} $n$ at $(\bar{x},\bar{y})$ if $T_{\gph F}(\bar{x},\bar{y})\in \Z_{n,m}$. By $\mathcal{O}_{F}$
 we denote the subset of $\gph F$, where $F$ is graphically smooth of dimension $n$.\\
\end{definition}
Clearly, for $(x,y)\in \mathcal{O}_{F}$ and $L=T_{\gph F}(x,y)=\gph DF(x,y)$ it holds that
$L^{\perp}=\hat{N}_{\gph F}(x,y)$ and $L^{*}=\gph \widehat{D}^{*}F(x,y)$.

Next we introduce the two derivative-like mappings
$\Sp F:\R^n \times  \R^m \tto \Z_{n,m}$ and $\Sp^* F:\R^n\times\R^m \tto \Z_{m,n}$ defined by
\begin{align*}
\Sp F(x,y):&=\{L \in \Z_{n,m} \mv \exists(x_{k},y_{k})\longsetto{\OO_F}(x,y) \mbox{ such that } \lim_{k\to\infty} d_{\Z_{n,m}}\big(L, T_{\gph F}(x_{k},y_{k})\big)=0\},
\end{align*}
and
\begin{align*}
{\mathcal{S}^{*}}F(x,y):&= \{L^*\mv L\in\Sp F(x,y)\}.
\end{align*}
Both $\Sp F$ and $\Sp^*F$ are in fact  generalized derivatives of $F$ whose elements, by virtue of the above definitions, are subspaces. Further we have the relation
\begin{gather*}
L^*\subset \gph D^{*}F(x,y)\ \forall L^*\in\Sp^*F(x,y),
\end{gather*}
cf. \cite[Eq 15]{GfrOut23}.

In what follows  $\Sp F$ will be called {\em  primal SC (subspace containing)  derivative} or {\em SC limiting graphical derivative} and $\Sp^* F$ will be termed
{\em adjoint SC  derivative (SC limiting coderivative)}. Let us mention that in \cite{GfrOut23} the SC adjoint derivative $\Sp^*F$ was defined in a different but equivalent manner.

Note that $\Sp F$ coincides  with the mapping $\cl \widehat \Sp F$, where
\[\widehat \Sp F(x,y):= T_{\gph F}(x,y),\ (x,y)\in \dom \widehat \Sp F:=\OO_F.\]
In particular, $\gph \Sp F$ is closed in $\R^n\times\R^m\times \Z_{n,m}$ and, due to the isometry $L \leftrightarrow L^*$, $\gph \Sp^* F$ is closed in $\R^n\times\R^m\times\Z_{m,n}$ as well.

On the basis of these SC derivatives we may now introduce the following notion playing a crucial role in the sequel.
\begin{definition}
A mapping
$F:\R^n\tto\R^m$ is said to have the {\em SCD property} at $(\xb,\yb)\in \gph F$, provided $\Sp^*F(\xb,\yb)\neq \emptyset$. $F$ is termed an {\em SCD mapping }if it has the SCD property at all points of $\gph F$.
\end{definition}
By virtue of the definition, the \SCD property at $(\xb,\yb)$ is obviously equivalent with the condition $\Sp F(\xb,\yb)\not=\emptyset$. Since we consider convergence in the compact metric space $\Z_{n,m}$ we readily obtain that $F$ has the \SCD property at $(\xb ,\yb)\in\gph F$ if and only if $(\xb,\yb)\in\cl \OO_F$.

\if{
Since we consider convergence in the compact metric space $\Z_{n,m}$, by using similar arguments as in the proof of \cite[Lemma 3.6]{GfrOut22a}, we readily obtain the following result. \todo{Brauchen wir das?}
\begin{lemma}\label{LemSCDproperty}
 A mapping $F:\R^n\tto\R^m$ has the \SCD property at $(x,y)\in\gph F$ if and only if $(x,y)\in\cl \OO_F$. Further, $F$ is an \SCD mapping if and only if $\cl \OO_F=\cl \gph F$, i.e., $F$ is graphically smooth of dimension $n$ at the points of a dense subset of its graph.
\end{lemma}
}\fi
\if{
\begin{example}
  Let $F:\R^n\tto\R^m$ and define the mappings $\widehat \Sp F:\R^n\times\R^m\tto \Z_{n,m}$ and $\widehat \Sp^* F:\R^n\times\R^m\tto \Z_{m,n}$ by
  \[\widehat \Sp F(x,y)=\begin{cases}
    T_{\gph F}(x,y)&\mbox{if $(x,y)\in \OO_F$,}\\\emptyset&\mbox{else,}
  \end{cases}\qquad
  \widehat \Sp^* F(x,y)=\begin{cases}
    T_{\gph F}(x,y)^*&\mbox{if $(x,y)\in \OO_F$,}\\\emptyset&\mbox{else.}
  \end{cases}\]
  Then for every $(x,y)\in\gph F$ there holds $\Sp F(x,y)=(\cl \widehat\Sp F)(x,y)$ and $\Sp^* F(x,y)=(\cl \widehat\Sp^* F)(x,y)$.
\end{example}
}\fi

For SCD mappings, the adjoint SC derivative $\Sp^*F$ forms some kind of skeleton of the limiting coderivative $D^*F$. It contains the information which is important in applications, in particular for numerical computations, but is much easier to compute than the limiting coderivative. On the other hand, the limiting coderivative is a more general object important for arbitrary set-valued mappings.

The derivatives $\Sp F$ and $\Sp^*F$ can be considered as a generalization of the B-Jacobian to multifunctions. If $f:\Omega\to\R^m$, $\Omega\subset\R^n$ open, is continuous, then for every $x\in \Omega$ there holds
  \begin{align*}
    &\Sp f(x):=\Sp\big(x,f(x)\big)\supseteq \{\rge(I,A)\mv A\in\overline{\nabla} f(x)\},\\
    &\Sp^* f(x):=\Sp^*\big(x,f(x)\big)\supseteq \{\rge(I,A^T)\mv A\in\overline{\nabla} f(x)\}.
  \end{align*}
  If $f$ is Lipschitz  near $x$, these inclusions hold with equality and $f$ has the \SCD property at all points $x'$ sufficiently close to  $x$, cf.\cite[Lemma 3.5]{GfrOut23}.

\if{In case of single-valued continuous mappings one has the following relationship.
\begin{lemma}[cf. {\cite[Lemma 3.5]{GfrOut23}}]
  Let $\Omega\subset \R^n$ be open and let $f:\Omega\to\R^m$ be continuous. Then for every $x\in \Omega$ there holds
  \begin{align*}
    &\Sp f(x):=\Sp\big(x,f(x)\big)\supseteq \{\rge(I,A)\mv A\in\overline{\nabla} f(x)\},\\
    &\Sp^* f(x):=\Sp^*\big(x,f(x)\big)\supseteq \{\rge(I,A^T)\mv A\in\overline{\nabla} f(x)\}.
  \end{align*}
  If $f$ is Lipschitz  near $x$, these inclusions hold with equality and $f$ has the \SCD property at all points $x'$ sufficiently close to  $x$.
\end{lemma}
}\fi

We refer to \cite{GfrOut22a,GfrOut23} for several calculus rules for SCD mappings.
\if{
For the sake of completeness we state also the following calculus rules. \todo{Prop.2.11, 2.12 brauchen wir eigentlich nicht}
\begin{proposition}[cf. {\cite[Theorem 4.1]{GfrOut23}}]\label{PropGraphLipsch}
  Given a mapping $G:\R^l\tto\R^m$ and a mapping $\Phi:\R^n\times\R^m\to\R^l\times\R^m$, consider the mapping $F:\R^n\tto\R^m$ given by
  \[\gph F=\{(x,y)\mv \Phi(x,y)\in\gph G\}.\]
  Further assume that we are given $(x,y)\in\gph F$ such that $\Phi$ is continuously differentiable in some neighborhood of $(x,y)$, $\nabla \Phi(x,y)$ has full row rank and $G$ has the SCD property at $\Phi(x,y)$. Then $F$ has the SCD property at $(x,y)$ and
  \begin{align*}
   & \Sp F(x,y)=\{L\in\Z_{n,m}\mv \nabla \Phi(x,y)L\in \Sp G(\Phi(x,y))\}\\
   & \Sp^*F(x,y)=S_{nm}\nabla \Phi(x,y)^TS_{lm}^T\Sp^* G(\Phi(x,y)).
   \end{align*}
\end{proposition}
\begin{proposition}\label{PropSumRule}
   Let $F:\R^n\tto\R^m$ have the \SCD property at $(x,y)\in\gph F$ and let $h:\Omega\to\R^m$ be continuously differentiable at $x\in \Omega$ where $\Omega\subseteq\R^n$ is open. Then $F+h$ has the \SCD property at $(x,y+h(x))$ and
  \begin{align*}
  \Sp (F+h)(x,y+h(x))=\Big\{\left(\begin{matrix}I_n&0\\\nabla h(x)&I_m\end{matrix}\right)L\mv L\in \Sp F(x,y)\Big\},\\
  \Sp^* (F+h)(x,y+h(x))=\Big\{\left(\begin{matrix}I_m&0\\\nabla h(x)^T&I_n\end{matrix}\right)L^*\mv L^*\in \Sp^* F(x,y)\Big\}.
  \end{align*}
\end{proposition}
\begin{proof}
  Follows from Proposition \ref{PropGraphLipsch} and the observation that $\gph (F+h)=\{(x,y)\mv (x,y-h(x))\in\gph F\}$.
\end{proof}
}\fi
An important class of SCD mappings is constituted by graphically Lipschitzian mappings:
\begin{definition}[cf. {\cite[Definition 9.66]{RoWe98}}]\label{DefGraphLip}
A mapping $F:\R^n\tto\R^m$ is {\em graphically Lipschitzian of dimension $d$} at $(\xb,\yb)\in\gph F$ if there is an open neighborhood $W$ of $(\xb,\yb)$ and a one-to-one mapping $\Phi$ from $W$ onto an open subset of $\R^{n+m}$ with $\Phi$ and $\Phi^{-1}$ continuously differentiable, such that $\Phi(\gph F\cap W)$  is the graph of a Lipschitz continuous mapping $f:U\to\R^{n+m-d}$, where $U$ is an open set in $\R^d$.
\end{definition}
\if{
As proved in \cite[Corollary 4.3]{GfrOut23}  each $F:\R^n\tto\R^m$, which is graphically Lipschitz of dimension $n$ at $(\xb,\yb)\in\gph F$, has the SCD property at $(\xb,\yb)$ and its SC generalized derivatives can be expressed in terms of $\nabla \Phi(\xb,\yb)$ and $\onabla f(\bar u)$, where $\big(\bar u,f(\bar u)\big)=\Phi(\bar x,\bar y)$.
}\fi
\begin{lemma}[cf. {\cite[Corollary 4.3]{GfrOut23}}]\label{LemSCDGraphLipsch}
 Assume that $F:\R^n\tto\R^m$ is graphically Lipschitzian of dimension $n$ at $(\xb,\yb)\in\gph F$. Then $F$ has
 the SCD property at $(\xb,\yb)$. Moreover, if $\Phi$ and $f$ are as in Definition \ref{DefGraphLip}, then
 \begin{align*}
   \Sp F(\xb,\yb&)=\{\nabla \Phi(\xb,\yb)^{-1}\rge(I,A)\mv A\in \onabla f(\bar u)\},\\
   \Sp^* F(\xb,\yb)&=\{ S_{nm}\nabla\Phi
(\bar{x},\bar{y})^{T}S^{T}_{nm}\rge(I,A^T)\mv A\in \onabla f(\bar u)\}\mbox{ with }S_{nm} :=
\begin{pmatrix}
0 & -I_{m}\\
I_{n} & 0
\end{pmatrix},
 \end{align*}
 where $\big(\bar u,f(\bar u)\big)=\Phi(\bar x,\bar y)$.
\end{lemma}

Prominent examples for graphically Lipschitzian mappings are the subdifferential mapping $\partial f$ for prox-regular and subdifferentially continuous functions $f:\R^n\to\oR$, cf. \cite[Proposition 13.46]{RoWe98} or, more generally, locally maximally hypomonotone mappings, cf. \cite{GfrOut22a}.

\subsection{\label{SubSecSS}On semismooth$^*$ mappings}

The following definition goes back to the recent paper \cite{GfrOut21}.

\if{
\begin{definition}\label{DefSemiSmooth}
\begin{enumerate}
\item  A set $A\subseteq\R^n$ is called {\em \ssstar} at a point $\xb\in A$ if for every $\epsilon>0$ there is some $\delta>0$ such that
  \begin{equation*}
\vert \skalp{x^*,x-\xb}\vert\leq \epsilon \norm{x-\xb}\norm{x^*}\ \forall x\in A\cap\B_\delta(\xb)\
\forall x^*\in N_A(x).
\end{equation*}
\item
A set-valued mapping $F:\R^n\tto\R^m$ is called {\em \ssstar} at a point $(\xb,\yb)\in\gph F$, if
$\gph F$ is \ssstar at $(\xb,\yb)$, i.e.,   for every $\epsilon>0$ there is some $\delta>0$ such that
\begin{align}
\label{EqCharSemiSmoothLim}&\lefteqn{\vert \skalp{x^*,x-\xb}-\skalp{y^*,y-\yb}\vert}\\
\nonumber&\quad\leq \epsilon
\norm{(x,y)-(\xb,\yb)}\norm{(y^*,x^*)}\ \forall(x,y)\in \gph F\cap\B_\delta(\xb,\yb)\ \forall
(y^*,x^*)\in\gph D^*F(x,y).
\end{align}
\end{enumerate}
\end{definition}
}\fi
\begin{definition}\label{DefSemiSmooth}
A set-valued mapping $F:\R^n\tto\R^m$ is called {\em \ssstar} at a point $(\xb,\yb)\in\gph F$, if for every $\epsilon>0$ there is some $\delta>0$ such that
\begin{align}
\label{EqCharSemiSmoothLim}&\lefteqn{\vert \skalp{x^*,x-\xb}-\skalp{y^*,y-\yb}\vert}\\
\nonumber&\quad\leq \epsilon
\norm{(x,y)-(\xb,\yb)}\norm{(y^*,x^*)}\ \forall(x,y)\in \gph F\cap\B_\delta(\xb,\yb)\ \forall
(y^*,x^*)\in\gph D^*F(x,y).
\end{align}
\end{definition}
Note that the semismooth$^*$ property was defined in \cite{GfrOut21} in a different but equivalent way. The class of semismooth$^*$ mappings is rather broad. For instance, every mapping whose graph is either the union of finitely many convex sets or a closed subanalytic set, is \ssstar at every point of its graph \cite[Proposition 2.10]{GfrOut22a}.

\if{We mention here two important classes of multifunctions having this property.
\begin{proposition}[cf. {\cite[Proposition 2.10]{GfrOut22a}}]\label{PropSSstar}
  \begin{enumerate}
    \item[(i)]Every mapping whose graph is the union of finitely many closed convex sets is \ssstar at every point of its graph.
    \item[(ii)]Every mapping with closed subanalytic graph is \ssstar at every point of its graph.
  \end{enumerate}
\end{proposition}
}\fi

A locally Lipschitz mapping $F:\Omega\to\R^m$, $\Omega \subset \R^n$ open, is \ssstar at $\xb$ if and only if $F$ is G-semismooth at $\xb$  \cite[Proposition 3.7]{GfrOut21}. However, there are also single-valued mappings which are semismooth$^*$ but neither locally Lipschitz and nor G-semismooth, e.g., the function $x\mapsto\sqrt{\vert x\vert}$  at $0$.

\if{functions there holds the following result.
{
\begin{proposition}[cf. {\cite[Proposition 3.7]{GfrOut21}}]\label{PropNewtondiff}
Assume that $F:\R^n\to \R^m$ is a single-valued mapping
which is Lipschitz near $\xb$. Then the following two statements are equivalent.
\begin{enumerate}
  \item[(i)] $F$ is \ssstar at $\xb$.
  \item[(ii)] For every $\epsilon>0$ there is some $\delta>0$ such that
  \begin{equation*}
    \norm{F(x)-F(\xb)-C(x-\xb)}\leq\epsilon\norm{x-\xb}\ \forall x\in \B_\delta(\xb)\ \forall
    C\in\co \onabla F(x),
  \end{equation*}
  i.e.,
  \begin{equation}\label{EqCharSemiSmoothLipsch1}
    \sup_{C\in \co\onabla F(x)}\norm{F(x)-F(\xb)-C(x-\xb)}=\oo(\norm{x-\xb}).
  \end{equation}
\end{enumerate}
\end{proposition}

Let us compare the semismooth$^*$ property with the semismooth property as introduced in \cite{QiSun93}.
\begin{definition}\label{DefSS}
  Let $\Omega\subseteq\R^n$ be nonempty and open. A function
$F:\Omega\to\R^m$ is {\em semismooth} at $\xb\in \Omega$, if it is  Lipschitz near $\xb$ and if
\[\lim_{\AT{A\in\co\onabla F(\xb+tu')}{u'\to u,\ t\downarrow 0}}Au'\]
exists for all $u\in\R^n$. If F is semismooth at all $\xb\in \Omega$, we call $F$ {\em semismooth on} $\Omega$.
\end{definition}

Then the following relation is well-known, see, e.g., \cite[Proposition 2.7]{Ulb11}.
\begin{proposition}
  Let $F:\Omega\to\R^m$ be defined on the open set $\Omega\subset\R^n$. Then for $\xb\in \Omega$ the following two statements are equivalent.
  \begin{enumerate}
    \item $F$ is semismooth at $\xb$,
    \item $F$ is Lipschitz near $\xb$, the directional derivative $F'(\xb;\cdot)$ exists and \eqref{EqCharSemiSmoothLipsch1} is fulfilled.
  \end{enumerate}
\end{proposition}
Hence, by virtue of Proposition \ref{PropNewtondiff}, for locally Lipschitz functions the difference between the semismooth$^*$ property and the semismooth property is just the directional differentiability. Further, we want to mention that under local Lipschitz continuity the semismooth property as defined by Gowda \cite{Gow04} coincides with our semismooth$^*$ property.
However, there are also single-valued functions which are semismooth$^*$ but not locally Lipschitz, e.g., the function $x\mapsto\sqrt{\vert x\vert}$  at $0$. This example demonstrates that even for single-valued functions the semismooth$^*$ property differs from the semismooth property more than only by the lack of directional differentiability.
}\fi

Sometimes it is useful to replace $\gph D^*F(x,y)$ in the relation \eqref{EqCharSemiSmoothLim}  by another set. \if{We will call a set-valued mapping $\Gamma:\R^p\tto\R^q$ {\em cone-valued} if for every $z\in \dom\Gamma$ the set $\Gamma(z)$ is a (not necessarily convex) cone. Note that the mapping $(x,y)\tto\gph D^*F(x,y)$ is cone-valued.}\fi

\begin{definition}
Let $F:\R^n\tto\R^m$ and $\Gamma:\R^n\times\R^m\tto\R^m\times \R^n$ be mappings and let $(\xb,\yb)\in\gph F$. We say that $F$ is {\em \ssstar with respect to $\Gamma$} (or shortly {\em $\Gamma$-ss$^*$}) at $(\xb,\yb)$, if the following three conditions are fulfilled:
\begin{enumerate}
  \item $\Gamma$ is cone-valued,
  \item there is some neighborhood $U$ of $(\xb,\yb)$ such that $\gph F\cap U\subset \dom \Gamma$,
  \item for every $\epsilon>0$ there is some $\delta>0$ such that
\begin{align*}
&\lefteqn{\vert \skalp{x^*,x-\xb}-\skalp{y^*,y-\yb}\vert}\\
&\qquad\leq \epsilon\norm{(x,y)-(\xb,\yb)}\norm{(y^*,x^*)}\ \forall(x,y)\in \gph F\cap\B_\delta(\xb,\yb)\ \forall (y^*,x^*)\in\Gamma(x,y).
\end{align*}
\end{enumerate}
We say that $F$ is $\Gamma$-ss$^*$ around $(\xb,\yb)$ if $F$ is $\Gamma$-ss$^*$ for all $(x,y)\in\gph F$ sufficiently close to $(\xb,\yb)$. Finally, $F$ is called $\Gamma$-ss$^*$ if $F$ is $\Gamma$-ss$^*$ at every $(x,y)\in\gph F$.
\end{definition}
Note that the semismooth$^*$ property introduced in Definition \ref{DefSemiSmooth} amounts to the semismooth$^*$ property of $F$ at $(\xb,\yb)$ with respect to $\Gamma(x,y)=\gph D^*F(x,y)$.

The notion of $\Gamma$-semismoothness$^*$ has two main applications. Firstly, the calculus rules for coderivatives often yield only inclusions of the form $D^*F(x,y)(y^*)\subset \Sigma(x,y)(y^*)$ with some closed-graph mapping $\Sigma$.
However, as in our motivating application, sometimes it is not necessary to compute the exact limiting coderivative but elements from $\Sigma$ related to the semismooth$^*$ property. Thus, if we can show that $F$ is $\Gamma$-ss$^*$ with respect to the mapping $\Gamma(x,y)=\gph\Sigma(x,y)$ then we can also work with $\Sigma$ instead of $D^*F$.

Secondly, the $\Gamma$-ss$^*$ property will be useful in connection with the theory of SCD mappings.
\begin{definition}
  We say that the mapping $F:\R^n\tto\R^m$ is SCD-ss$^*$  at (around)  $(\xb,\yb)\in\gph F$  if it is semismooth$^*$  at (around) $(\xb,\yb)$ with respect to the mapping
  $\bigcup\Sp^*F: \R^n\times\R^m\tto \R^m\times\R^n$ given by
\[(\bigcup\Sp^*F)(x,y):=\bigcup_{L^*\in\Sp^* F(x,y)} L^*.\]
The mapping $F$ is called SCD-ss$^*$  if it is SCD-ss$^*$  at every point of its graph.
\end{definition}
It follows from the definition of the $\Gamma$-ss$^*$ property that $F$ has the SCD property around $(\xb,\yb)$ whenever $F$ is SCD-ss$^*$ at $(\xb,\yb)$.

The SCD-ss$^*$ property has also a useful primal interpretation:

\begin{proposition}[cf. {\cite[Proposition 2.12]{Gfr24}}]\label{PropSCD_ss_Primal}
Assume that $F:\R^n\tto\R^m$ has the SCD property around  the  pair $(\xb,\yb)\in\gph F$. Then the following two statements are equivalent:
  \begin{enumerate}
    \item $F$ is SCD-ss$^*$ at $(\xb,\yb)$.
    \item
    $\limsup\limits_{(x,y)\longsetto{\gph F}(\xb,\yb)}\ \sup\limits_{L\in \Sp F(x,y)}{\dist{(x-\xb,y-\yb),L}}/{\norm{(x-\xb,y-\yb)}} =0.$
   \if{ For every $\epsilon>0$ there is some $\delta>0$ such that for every $(x,y)\in \gph F\cap\B_\delta(\xb,\yb)$ and every $L\in \Sp F(x,y)$ there holds \todo{Schreiben als limsup}
    \[\dist{(x-\xb,y-\yb),L}\leq\epsilon\norm{(x-\xb,y-\yb)}.\]
    }\fi
  \end{enumerate}
\end{proposition}

\section{On $\Psi$-semismooth single-valued mappings}

The following definition is closely related to the definition of semismooth operators in Banach spaces, cf. \cite[Definition 3.1]{Ulb11}.
\begin{definition}\label{DefSS_Psi}
  Consider the mapping $F:\Omega\to\R^m$, where $\Omega\subset\R^n$ is open, and the multifunction $\Psi:\R^n\tto\R^{m\times n}$.
  \begin{enumerate}
  \item We say that $F$ is {\em semismooth  with respect to $\Psi$} (shortly, $\Psi$-ss) at $\xb\in\Omega$ if $\dom \Psi$ is a neighborhood of $\xb$, $\Psi$ is locally bounded at $\xb$ and
    \begin{equation}\label{EqGsemismooth}
      \lim_{\AT{x\setto{\Omega} \xb}{x\not=\xb}}\frac{\sup\{\norm{F(x)-F(\xb)-A(x-\xb)}\mv A\in\Psi(x)\}}{\norm{x-\xb}}=0.
    \end{equation}
    \item We say that $F$ is  $\Psi$-ss  on a subset $U\subset\Omega$ if $F$ is $\Psi$-ss at every $x\in U$. The mapping $\Psi$ is called a {\em semismooth derivative} of $F$ on $U$.
    \end{enumerate}
\end{definition}

 By this definition, a locally Lipschitz mapping $F$ is $\co\onabla F$-ss if and only if it is G-semismooth.

Our definition of semismoothness differs from Ulbrich's one \cite[Definition 3.1]{Ulb11} insofar as we do not suppose local continuity of $F$ but require local boundedness of $\Psi$.
 If $F$ is $\Psi$-ss at $\xb$  then $\Psi$ is also called a {\em Newton map} for $F$ at $\xb$ in \cite{KlKu02}. If $\Psi$ is single-valued then it is also called a Newton derivative (or slanting function) for $F$ at $\xb$ in the literature.  We dispense with using this terminology because they are normally used at single points, namely solutions of the nonsmooth equation $F(x)=0$. However, for our purposes it is essential that the properties hold  not only at one single point $\xb$ but also  on a neighborhood of $\xb$. For instance, for solving \eqref{EqBasProbl} with the BT-algorithm, we will show below that any semismooth derivative of $\theta$ on an open superset of $\Uad$ is a feasible choice for the oracle.

 In the literature one can find also several papers about the $\Psi$-ss property under some additional assumptions on $\Psi$: For instance, if $F$ is $\Psi$-ss at $\xb$, $\Psi(x)$ is compact for all $x$ close to $\xb$ and $\Psi$ is usc at $\xb$, then $F$ is called {\em C-differentiable} at $\xb$ and $\Psi$ is called a {\em C-differential operator} of $F$ \cite{Qi96}.

Though we do not explicitly require (Lipschitz) continuity of $F$ in Definition \ref{DefSS_Psi} , it is implied by \eqref{EqGsemismooth} together with the boundedness of $\Psi$. It is easy to see that a mapping $F:\Omega\to\R^m$, $\Omega\subset\R^n$ open, which is $\Psi$-ss at $\xb\in\Omega$, is calm at $\xb$ and $\clm F(\xb)\leq\bnd \Psi(\xb)$. Further, as we will show below in Theorem \ref{ThNorkinR_m}, if $F$ is $\Psi$-ss on $\Omega$, then it is also locally Lipschitz there. The main reason for omitting the requirement of Lipschitz continuity in Definition \ref{DefSS_Psi} is that we want to deal also with mappings which are implicitly defined by solution mappings of generalized equations, where verifying Lipschitz continuity with conventional tools like limiting coderivatives might be a difficult issue.

There is a big difference between pointbased semismoothnes and semismoothness on a neighborhood. For instance, given any mapping $F:\Omega\to\R^m$, $\Omega\subset\R^n$ open, and $\xb\in \Omega$, the mapping $\Psi:\R^n\tto\R^{m\times n}$ given by $\dom\Psi=\Omega$ and
\[\Psi(x)=\begin{cases}\frac{(F(x)-F(\xb))(x-\xb)^T}{\norm{x-\xb}^2}&\mbox{if $x\in\Omega\setminus\{\xb\}$,}\\
0&\mbox{if $x=\xb$}\end{cases}\]
obviously satisfies \eqref{EqGsemismooth}. If $F$ is  calm at $\xb$ then  $\Psi$ is also locally bounded at $\xb$ and therefore $F$ is $\Psi$-ss at $\xb$. However, $F$ will not be $\Psi$-ss on any neighborhood of $\xb$, except for very special cases. As an example consider $F:\R\to\R$, $F(x):=x^2$ and $\xb:=0$. The formula above gives $\Psi(x)=x$ and
\[\lim_{\AT{x\to x'}{x\not= x'}}\frac{\norm{F(x)-F(x')-\Psi(x)(x-x)}}{\norm{x-x'}} = \lim_{\AT{x\to x'}{x\not= x'}}\frac{\vert x^2-x'^2-x(x-x')\vert}{\vert x-x'\vert}=\vert x'\vert\not=0\]
for every $x'\not=\xb$.

Moreover, in the pointbased setting we may modify a semismooth derivative $\Psi$ by adding matrices of order $\OO(\norm{F(x)-F(\xb)})$
to $\Psi(x)$ without violating \eqref{EqGsemismooth}. This issue is sometimes exploited in  semismooth Newton methods for finding zeros of $F$ by neglecting terms of order $\OO(\norm{F(x)})$, see, e.g. \cite{PaBe13}. However, with such a modification the semismoothness property on a neighborhood will no longer hold.

If a mapping $F:\Omega\to\R^m$, $\Omega\subset \R^n$ open, is $\Psi$-ss at $\xb\in\Omega$, it follows immediately from the definition that each component $F_i$, $i=1,\ldots,m$, is semismooth at $\xb$ with respect to
\begin{equation}\label{EqPsi_ssComponent}
\Psi_i(x):=\{e_i^TA\mv A\in\Psi(x)\},
\end{equation}
where $e_i$ denotes the $i$-th unit vector. Conversely. if each component $F_i$, $i=1,\ldots,m$, is semismooth at $\xb\in\Omega$ with respect to some mapping $\Psi_i:\R^n\to\R^{1\times n}$, then $F$ is semismooth  at $\xb$ with respect to
\begin{equation}\label{EqPsi_from_Components}
\Psi(x):=\{A\in\R^{m\times n}\mv e_i^TA\in\Psi_i(x)\}.
\end{equation}
This construction resembles the one of Qi's C-subdifferential, see, e.g. \cite[Proposition 2.1]{Ulb11}.

 Clearly, continuously differentiable mappings $F$  are $\nabla F$-ss.

It is well-known that Newton maps satisfy a chain rule, see, e.g., \cite[Theorem 6.14]{KlKu02}. For the sake of completeness we state the following result, where in contrast to \cite{KlKu02} we do not assume Lipschitz continuity of the involved mappings $F_1,F_2$:
\begin{lemma}\label{LemChainRule}
   Consider the  mappings $F_1:U\to\R^p$, $F_2:V\to\R^m$, where $U\subset\R^n$ and $V\subset\R^p$ are open sets with $F_1(U)\subset V$. Assume that $F_1$ is semismooth at $\xb\in U$  with respect to the mapping $\Psi_1:\R^n\tto \R^{p\times n}$ and that $F_2$ is semismooth at $F_1(\xb)$ with respect to the mapping $\Psi_2:\R^p\tto\R^{m\times p}$. Then  the composite mapping $F:U\to\R^m$ given by  $F(x):=F_2(F_1(x))$ is $\Psi$-ss  at $\xb$ with  $\Psi$ given by
  \[\Psi(x):=\{BA\mv A\in\Psi_1(x),\ B\in \Psi_2(F_1(x))\}.\]
\end{lemma}
\begin{proof}By calmness of $F_1$ at $\xb$ and calmness of $F_2$ at $F_1(\xb)$   we can find  neighborhoods $\tilde U\subset U$ of $\xb$  and $\tilde V\subset V$ of $F_1(\xb)$ together with two positive reals $\kappa_1,\ \kappa_2$ such that $\norm{F_1(x)-F_1(\xb)}\leq\kappa_1\norm{x-\xb}$ $\forall x\in\tilde U$, $\norm{F_2(y)-F_2(F_1(\xb))}\leq \kappa_2\norm{y-F_1(\xb)}$ $\forall y\in \tilde V$ and  $F_1(\tilde U)\subset \tilde V$. By possibly shrinking $\tilde U$, $\tilde V$ and enlarging $\kappa_1,\kappa_2$ we may also assume that $\bnd \Psi_1(x)\leq\kappa_1$ $\forall x\in \tilde U$ and $\bnd \Psi_2(y)\leq\kappa_2$ $\forall y\in\tilde V$. Next, for arbitrary $\epsilon>0$ we choose a neighborhood $V_\epsilon\subset\tilde V$ of $F_1(\xb)$ such that for each $y\in V_\epsilon$ and each $B\in \Psi_2(y)$ one has $\norm{F_2(y)-F_2(F_1(\xb))-B(y-F_1(\xb))}\leq\epsilon\norm{y-F_1(\xb)}$. Finally we choose a neighborhood $U_\epsilon\subset\tilde U$ such that $F_1(U_\epsilon)\subset V_\epsilon$ and for every $x\in U_\epsilon$ and every $A\in \Psi_1(x)$ one has $\norm{F_1(x)-F_1(\xb)-A(x-\xb)}\leq\epsilon\norm{x-\xb}$. Thus, for every $x\in U_\epsilon$, every $A\in\Psi_1(x)$ and every $B\in \Psi_2(F_1(x))$ we obtain
\begin{align*}
  &\lefteqn{\norm{F(x)-F(\xb)-BA(x-\xb)}}\\
  &\leq \norm{F_2(F_1(x))-F_2(F_1(\xb))-B(F_1(x)-F_1(\xb))}+\norm{B\big(F_1(x)-F_1(\xb)-A(x-\xb)\big)}\\
  &\leq \epsilon\norm{F_1(x)-F_1(\xb)}+\kappa_2\epsilon\norm{x-\xb}\leq\epsilon(\kappa_1+\kappa_2)\norm{x-\xb}.
\end{align*}
 In order to show the local boundedness of $\Psi$ at $\xb$ just note that for every $x\in\tilde U$ and every $BA\in \Psi(x)$ we have $\norm{BA}\leq\norm{B}\norm{A}\leq\kappa_2\kappa_1$. This proves that $F$ is $\Psi$-ss at $\xb$.
\end{proof}
Using similar arguments one can prove the following lemma.
\begin{lemma}
Consider an open set $\Omega\subset\R^n$, two mappings $F_i:\Omega\to\R^m$ and two mappings $\Psi_i:\R^n\tto\R^m$, $i=1,2$. If $F_i$ is $\Psi_i$-ss, $i=1,2$, at $\xb\in\Omega$ then $F_1+F_2$ is $(\Psi_1+\Psi_2)$-ss at $\xb$.
\end{lemma}
The next auxiliary statement is important in the application we have in mind.
\begin{lemma}\label{LemNorkin}Let $F:\Omega\to\R^m$, $\Omega\subset\R^n$ open, be continuous and assume that $F$ is semismooth at $\xb\in\Omega$ with respect to $\Psi:\R^n\tto\R^{m\times n}$. Then  $F$ is also $\cl\Psi$-ss and $\cocl\Psi$-ss at $\xb$.
\end{lemma}
\begin{proof}
  Let  $\epsilon>0$ be arbitrarily fixed.  By definition we can find some $\delta>0$ such that there holds
  \[\norm{F(x)-F(\xb)-A(x-\xb)}\leq\epsilon\norm{x-\xb}\ \forall x\in\B_\delta(\xb)\ \forall A\in\Psi(x).\]
  For every  $x\in \B_{\delta/2}(\xb)\setminus\{\xb\}$ and $A\in(\cl\Psi)(x)$  we can find some sequence $(x_k,A_k)\longsetto{\gph\Psi}(x,A)$. Since $x_k\in\B_\delta(\xb)$ for all $k$ sufficiently large, we can conclude
  \[\frac{\norm{F(x)-F(\xb)-A(x-\xb)}}{\norm{x-\xb}}=\lim_{k\to\infty} \frac{\norm{F(x_k)-F(\xb)-A_k(x_k-\xb)}}{\norm{x_k-\xb}}<\epsilon,\]
  verifying that $F$ is $\cl\Psi$-ss at $\xb$. The $\cocl\Psi$-ss property of $F$ at $\xb$ follows now easily from the fact  that any matrix $A\in(\cocl\Psi)(x)$ can be written as  convex combination $A=\sum_{i=1}^N\lambda_iA_i$ with $A_i\in(\cl\Psi)(x)$ and $\sum_{i=1}^N\lambda_i=1$, $\lambda_i\geq0$, $i=1,\ldots,N$.
\end{proof}
\begin{remark}As a consequence of this lemma we obtain that a continuous  mapping $F:\Omega\to\R^m$, $\Omega\subset \R^n$ open, which is $\Psi$-ss at $\xb\in\Omega$, is also C-differentiable at $\xb$ with C-differential operator $\cl\Psi$ or $\cocl\Psi$. Further, for any sequence $x_k\setto{\Omega}\xb$ we can select a bounded sequence $A_k\in\Psi(x_k)$ and, by passing to a subsequence $x_{k_j}$ such that $A_{k_j}$ converges to some $A\in(\cl\Psi)(\xb)$, we obtain that
\[F(x_{k_j})-F(\xb)-A(x_{k_j}-\xb)=\oo(\norm{x_{k_j}-\xb}),\]
i.e., $F$ is {\em H-differentiable} at $\xb$ with H-differential $(\cl\Psi)(\xb)$ in the sense of \cite{GoRa00}.
\end{remark}

Consider now the following definition.
\begin{definition}[cf. {\cite{No78,No80}}]
  A function $f:\R^n\to\R$ is called {\em generalized differentiable}  if there exists an usc  mapping $G:\R^n\tto\R^n$ with nonempty, convex and compact values $G(x)$ such that
  for every $\xb\in\R^n$ there holds
    \[\sup_{g\in G(x)}\vert f(x)-f(\xb)-\skalp{g,x-\xb}\vert=\oo(\norm{x-\xb}).\]
  The elements of $G(x)$ are called {\em pseudogradients}.
\end{definition}
By  Lemma \ref{LemNorkin}, a $\Psi$-ss function $f:\R^n\to\R$ is also generalized differentiable with
\[G(x)=\{g\in\R^n\mv g^T\in(\cocl\Psi)(x)\}.\]
Note that pseudogradients are column vectors whereas the elements of $\Psi$ are $1\times n$ matrices, i.e., row vectors.

 Generalized differentiable functions have a lot of interesting properties, cf. \cite{No78,No80,MiGuNo87}, which can be carried over to $\Psi$-semismooth functions $f:\Omega\to\R$, $\Omega\subset\R^n$ open. For the sake of completeness we state some of these properties in the following proposition.

\begin{proposition}\label{PropNorkin}
  Let $\Omega\subset\R^n$ be open and let $f:\Omega\to\R$ be semismooth on $\Omega$ with respect to the mapping $\Psi:\R^n\tto\R^{1\times n}$. Then the following statements hold true.
  \begin{enumerate}
    \item[(i)] $f$ is locally Lipschitz on $\Omega$ and for every $x\in\Omega$  there holds $\partial^cf(x)^T\subset(\cocl\Psi)(x)$.
    \item[(ii)] Let $\Omega_\Psi$ denote the set of all $x\in\Omega$ such that $(\cocl\Psi)(x)$ is a singleton. Then $\Omega\setminus\Omega_\Psi$ has Lebesgue measure zero. Moreover, the function $f$ is strictly differentiable at every $x\in\Omega_\Psi$ with $\{\nabla f(x)\}=(\cocl\Psi)(x)=\Psi(x)$.
  \end{enumerate}
\end{proposition}
\begin{proof}
  ad (i): The assertions have been shown in \cite[Theorems 1.1, 1.10]{MiGuNo87} for the special case $\Omega=\R^n$ and it is an easy task to carry over the proofs to general open sets $\Omega$.

  ad (ii): In order to show that the set $\Omega\setminus\Omega_\Psi$ is negligible, one can take over the proof of  \cite[Theorem 1.12]{MiGuNo87} to show that for every compact cube $K(x,\delta):=\{x'\in\R^n\mv\max_i\vert x'_i-x_i\vert\leq\delta\}\subset \Omega$ the set $K(x,\delta)\setminus\Omega_\Psi$ is negligible. Consider the countable collection $K(x_1,\delta_1), K(x_2,\delta_2),\ldots$ of cubes contained in $\Omega$ with rational $x_i$ and rational $\delta_i$. This collection covers $\Omega$ and therefore the set $\Omega\setminus\Omega_\Psi$ is negligible as the countable union of the negligible sets $K(x_i,\delta_i)\setminus \Omega_\Psi$.

  For every $x\in\Omega_\Psi$ we conclude from (i) that $\partial^cf(x)=(\cocl \Psi)(x)=\Psi(x)$ is a singleton and thus $f$ is strictly differentiable at $x$ with $\partial^c f(x)=\{\nabla f(x)\}$ by \cite[Proposition 2.2.4]{Cla83}.
  \end{proof}
\if{
\begin{proposition}\label{PropNorkin}
  Let $\Omega\subset\R^n$ be open and let $f:\Omega\to\R$ be semismooth on $\Omega$ with respect to the mapping $\Psi:\R^n\tto\R^{1\times n}$. Then the following statements hold true.
  \begin{enumerate}
    \item[(i)] For every $x\in\Omega$  there holds $\widehat\partial f(x)^T\subset(\cocl\Psi)(x)$.
    \item[(ii)] Let  $x,x'\in\Omega$ be such that the line segment $[x,x']$ connecting $x$ and $x'$ is contained in $\Omega$. Then there exists $\bar t\in (0,1)$ and $g\in (\cocl\Psi)\big(x+\bar t(x'-x)\big)$ with $f(x')-f(x)=g(x'-x)$.
    \item[(iii)] $f$ is locally Lipschitz on $\Omega$ and for $x\in\Omega$ there holds $\lip f(x)\leq\bnd\Psi(x)$.
    \item[(iv)] For all $x\in\Omega$ there holds $\emptyset\not=\partial f(x)^T\subset\partial^cf(x)^T\subset (\cocl\Psi)(x)$.
    \item[(v)] Let $\Omega_\Psi$ denote the set of all $x\in\Omega$ such that $(\cocl\Psi)(x)$ is a singleton. Then $\Omega\setminus\Omega_\Psi$ has Lebesgue measure zero. Moreover, the function $f$ is strictly differentiable at every $x\in\Omega_\Psi$ with $\{\nabla f(x)\}=(\cocl\Psi)(x)$.
  \end{enumerate}
\end{proposition}
\begin{proof}
  ad (i): Assume on the contrary that there exist $x\in\Omega$ and $g^T\in\widehat\partial f(x)$ with $g\not\in (\cocl\Psi)(x)$ and consider the projection $\bar g$ of $g$ onto the convex compact set $(\cocl\Psi)(x)$, which must be different from $g$. Pick any sequence $t_k\downarrow 0$ and consider the points $x_k:=x+t_k(g-\bar g)^T$. For sufficiently large $k$ we have $x_k\in\Omega$ and we can pick some $g_k\in\Psi(x_k)$. Then the sequence $g_k$ is bounded and, by possibly passing to a subsequence, we can assume that $g_k$ converges to some $\hat g$. Clearly, $\hat g\in(\cl\Psi)(x)\subset(\cocl\Psi)(x)$ and therefore $(g-\bar g)(\hat g-\bar g)^T\leq 0$.  Since $f$ is $\Psi$-ss on $\Omega$, we obtain that
  \[\lim_{k\to\infty}\frac{f(x_k)-f(x)-g_k(x_k-x)}{\norm{x_k-x}}=0.\]
  Further we have
  \[\liminf_{k\to\infty}\frac{f(x_k)-f(x)-g(x_k-x)}{\norm{x_k-x}}\geq 0\]
  due to $g^T\in\widehat\partial f(x)$ and consequently
  \[0\leq\liminf_{k\to\infty}\frac{(g_k-g)(x_k-x)}{\norm{x_k-x}}=\frac{(\hat g -g)(g-\bar g)^T}{\norm{g-\bar g}}=\frac{(\hat g -\bar g)(g-\bar g)^T}{\norm{g-\bar g}}-\norm{g-\bar g}<0,\]
  a contradiction.

  ad(ii):  By Lemma \ref{LemSSCalm}, $f$ is calm at every $x\in\Omega$ and therefore continuous. Let $\I:=\{t\in\R\mv x+t(x'-x)\in\Omega\}\supset[0,1]$ and consider the function $\varphi:\I\to\R$, $\varphi(t):=f(x+t(x'-x))-f(x)-t(f(x')-f(x))$. Owing to Lemmas \ref{LemChainRule}, \ref{LemSumRule}, $\varphi$ is semismooth on $\I$ with respect to the mapping $\Phi(t):=\Psi(x+t(x'-x))(x'-x)-(f(x')-f(x))$, $t\in \I $. There holds $\varphi(0)=\varphi(1)=0$ and, since $\varphi$ is continuous on $[0,1]$, there exists $\bar t\in(0,1)$ being either a global minimizer or a global maximizer of $\varphi$ on $[0,1]$. We may assume that $\bar t$ is a local minimizer, since otherwise we can interchange $x$ and $x'$ to obtain a global minimizer $(1-\bar t)$ for the function $-\varphi(1-t)$. Then
  \[0\in\widehat\partial \varphi(\bar t) \subset (\cocl \Phi)(\bar t)=(\cocl \Psi)(x+\bar t(x'-x))(x'-x)-(f(x')-f(x))\]
 yielding the assertion.

  ad(iii): This is an immediate consequence of (ii).

  ad(iv): Since $f$ is locally Lipschitz on $\Omega$,  for every $x\in\Omega$ there holds $\partial f(x)=(\cl\widehat\partial f)(x)$ by definition of $\partial f$ and
  $\partial^cf(x)=\co \partial f(x)$ by \cite[Theorem 9.61]{RoWe98}. This yields, together with (i),
  \[\partial f(x)\subset\partial^c f(x)=(\cocl \widehat\partial f)(x)\subset\{g^T\mv g\in(\cocl(\cocl\Psi))(x)\}=\{g^T\mv g\in(\cocl\Psi)(x)\}.\]

  ad (v): In order to show that the set $\Omega\setminus\Omega_\Psi$ is negligible, one can take over the proof of  \cite[Theorem 1.12]{MiGuNo87} to show that for every compact cube $K(x,\delta):=\{x'\in\R^n\mv\max_i\vert x'_i-x_i\vert\leq\delta\}\subset \Omega$ the set $K(x,\delta)\setminus\Omega_\Psi$ is negligible. Consider the countable collection $K(x_1,\delta_1), K(x_2,\delta_2),\ldots$ of cubes contained in $\Omega$ with rational $x_i$ and rational $\delta_i$. This collection covers $\Omega$ and therefore the set $\Omega\setminus\Omega_\Psi$ is negligible as the countable union of the negligible sets $K(x_i,\delta_i)\setminus \Omega_\Psi$.

  For every $x\in\Omega_\Psi$ we conclude from (iv) that $\partial^cf(x)$ is a singleton and thus $f$ is strictly differentiable at $x$ with $\nabla f(x)=\{\partial^c f(x)\}$ by \cite[Proposition 2.2.4]{Cla83}. Since strict differentiability of $f$ is equivalent with continuity of $\nabla f$ relative to $\OO_f$, cf. \cite[Exercise 9.64]{RoWe98}, $\nabla f$ is continuous on $\Omega_\Psi$.
  \end{proof}
}\fi
\if{
It is not difficult to extend these results to vector-valued mappings.
}\fi
We are now in the position to state the main result of this section.
\begin{theorem}\label{ThNorkinR_m}
  Let $\Omega\subset\R^n$ be open and assume that $F:\Omega\to\R^m$ is semismooth on $\Omega$ with respect to a mapping $\Psi:\R^n\tto\R^{m\times n}$. Then the following statements hold:
  \begin{enumerate}
    \item[(i)] $F$ is locally Lipschitz on $\Omega$.
    \item[(ii)] For every $x\in\Omega$ there holds $\emptyset\not=\co\onabla F(x)\subset (\cocl\Psi)(x)$.
    \item[(iii)] Let $\Omega_\Psi$ denote the set of all $x\in\Omega$ such that $(\cocl\Psi)(x)$ is a singleton. Then $\Omega\setminus\Omega_\Psi$ has Lebesgue measure zero. Moreover,  the mapping $F$ is strictly differentiable at every $x\in \Omega_\Psi$ with $\co\onabla F(x)=\{\nabla F(x)\}=(\cocl\Psi)(x)=\Psi(x)$.
  \end{enumerate}
\end{theorem}
\begin{proof}
   ad (i): Each component $F_i$, $i=1,\ldots,m$, is semismooth on $\Omega$ with respect to $\Psi_i$ given by \eqref{EqPsi_ssComponent} and therefore locally Lipschitz on $\Omega$ by Proposition \ref{PropNorkin}. Hence, $F$ is locally Lipschitz on $\Omega$ as well.
   \if{
   By Lemma \ref{LemChainRule} we obtain that each component function $F_i$, $i=1,\ldots,m$, is semismooth on $\Omega$ with respect to the mapping $\Psi_i:\Omega\tto\R^{1\times n}$ given by $\Psi_i(x):=\{e_i^TA\mv A\in \Psi(x)\}$. Hence, the components $F_i$, $i=1,\ldots,m$ are locally Lipschitz on $\Omega$ and so is the mapping $F$ as well.
   }\fi

   ad (ii): Let $\Pi:\R^n\tto\R^{m\times n}$ denote the mapping which assigns to each $x$ the set of all $m\times n$ matrices  whose $i$-th row belongs to $(\cocl\Psi_i)(x)$. Clearly,
   $(\cocl\Psi)(x)\subset \Pi(x)$, $x\in\R^n$. Next let $\tilde \Omega:=\bigcap_{i=i}^m\Omega_{\Psi_i}$. For every $x\in\tilde \Omega$ we obtain from Proposition \ref{PropNorkin}(ii) that for every $i=1,\ldots,m$ the component function  $F_i$ is strictly differentiable at $x$ and  $(\cocl \Psi_i)(x)=\{\nabla F_i(x)\}$ is a singleton. Hence, $F$ is also strictly differentiable at $x$ and $(\cocl \Psi)(x)=\Pi(x)=\{\nabla F(x)\}$. Since the set $\Omega\setminus\tilde\Omega=\bigcup_{i=1}^m(\Omega\setminus\Omega_{\Psi_i})$ is negligible, the set $\tilde \Omega$ differs from the set $\OO_F$, the set of all points where $F$ is differentiable, by a negligible set. By defining the mapping $\nabla_sF(x):=\nabla F(x)$ for $x\in\dom\nabla_sF:=\tilde \Omega$
    and taking into account $\co\onabla F=\cocl \nabla F$, it follows from \cite[Theorem 4]{Wa81} that $\co\onabla F=\cocl \nabla_s F$. Since $\nabla_sF(x)\subset(\cocl \Psi)(x)$ $\forall x\in\Omega$, we obtain $\co\onabla F\subset\big(\cocl(\cocl\Psi)\big)(x)=(\cocl\Psi)(x)$ $\forall x\in\Omega$.

   ad (iii): Since $\tilde \Omega\subset\Omega_\Psi\subset\Omega$ and the set $\Omega\setminus \tilde\Omega$ is negligible, so must be the set $\Omega\setminus\Omega_\Psi$ as well. For every $x\in\Omega_\Psi$ we deduce from (ii) that $\co\onabla F(x)$ is also a singleton and therefore $F$ is strictly differentiable at $x$ by Lemma \ref{LemStrictDiff}.
\end{proof}
\begin{corollary}
  Let $F:\Omega\to\R^m$ be a mapping, where $\Omega\subset\R^n$ is open. Then the following two statements are equivalent and entail that $F$ is strictly differentiable on $\Omega$ up to a set of Lebesgue measure zero.
  \begin{enumerate}
    \item $F$ is semismooth on $\Omega$ with respect to some mapping $\Psi:\R^n\tto\R^{m\times n}$.
    \item $F$ is locally Lipschitz and $\co\onabla F$-ss on $\Omega$, i.e., G-semismooth.
  \end{enumerate}
\end{corollary}

Let us return to  problem \eqref{EqBasProbl} under the additional assumption that $\Uad$ is a convex polyhedron. The latter condition enables us to solve \eqref{EqBasProbl} numerically via the BT-algorithm which admits only $\Uad$ given by affine linear inequalities and equalities.

Assume that $\theta$ is locally Lipschitz and suppose that at each iteration of the BT-algorithm instead of the oracle \eqref{EqOracle} one uses another
  \begin{equation*}
  \begin{minipage}{10cm}{\em oracle}, which returns for arbitrary $x\in\R^n$ the function value $\theta(x)$ and one element $g(x)$ with $g(x)^T\in\Psi(x)$,\end{minipage}\end{equation*}
  where $\Psi:\R^n\tto\R^{1\times n}$ is a locally bounded mapping with $\dom\Psi$ being an open superset of $\Uad$. Further assume that the condition
\begin{equation}\label{EqWeakSSPsi}\lim_{t\downarrow 0}\frac{\sup\{\vert\theta(x+td)-\theta(x)-t\skalp{g,d}\vert\mv g\in\Psi(x+td)\}}t=0\ \forall x\in\R^n\ \forall d\in\R^n\end{equation}
is fulfilled, which ensures that a certain line-search procedure in the BT-algorithm is finite. Then, as already observed by \cite[Theorem 3.2]{DeBa01}, by replacing $\partial^c \theta$ with $\cocl \Psi$ in the convergence proof of the BT-algorithm, one can show the following result: If the sequence $z_k$ produced by the BT-algorithm remains bounded, we may conclude  that there exists an accumulation point $\zb$ of $z_k$ such that $0\in (\cocl \Psi)(\zb)^T+N_{\Uad}(\zb)$. Further, the above statement remains essentially true for other standard bundle methods.

If $\theta$ is $\Psi$-ss then condition \eqref{EqWeakSSPsi} is clearly fulfilled and the condition $0\in (\cocl \Psi)(\zb)+N_{\Uad}(\zb)$ is in fact a necessary optimality condition for a local minimizer because, by Proposition \ref{PropNorkin}, we have $\partial^c \theta(\zb)\subset(\cocl \Psi)(\zb)$. We claim (without proof) that the latter assertion also holds true under the weaker condition \eqref{EqWeakSSPsi}. Further, condition \eqref{EqWeakSSPsi} may be replaced by any other condition on $\Psi$ guaranteeing the finiteness of the line search procedure in  order  that the above convergence result of the BT-algorithm holds true. For instance, one could also use a condition related to upper semidifferentiability \cite{Bih84}.  These weakened versions of semismoothness, however,  will not be pursued in this paper because they do not satisfy a chain rule and are incompatible with the \ssstar property of multifunctions.

\begin{remark}\label{RemWeaklySS}As already mentioned in the introduction, in the original versions of the BT-algorithm \cite{SZ, Zo89} it is assumed that the objective $\theta$ is  weakly semismooth, which may be characterized by  directional differentiability of $\theta$ and the requirement that \eqref{EqWeakSSPsi} holds with $\Psi=\partial^c\theta$.

\end{remark}

\section{On strict proto-differentiability and the SCD-ss$^*$ property of graphically Lipschitzian mappings}
At first we state the relationship between strict proto-differentiability of a mapping and the property that its SC derivative is a singleton.
\begin{theorem}\label{ThStrictProtoDiff}
  Assume that the mapping $F:\R^n\tto\R^m$ has the SCD property at $(\xb,\yb)\in\gph F$. If $F$ is strictly proto-differentiable at $(\xb,\yb)$ and $\gph F$ is locally closed at $(\xb,\yb)$ then $\Sp F(\xb,\yb)$ is a singleton.
 If, in addition, $F$ is graphically Lipschitzian of dimension $n$ at $(\xb,\yb)$ then the reverse implication holds true as well.
\end{theorem}
\begin{proof}
Assume that $F$ is strictly proto-differentiable, then we know by Lemma \ref{LemStrictProto} that $\bar L:=T_{\gph F}(\xb,\yb)$ is a subspace. Further, from \eqref{EqTangStrictProto} we conclude that
\[\bar L\supset\Limsup_{(x,y)\longsetto{\OO_F}(\xb,\yb)}T_{\gph F}(x,y)\supset\bigcup_{L\in\Sp F(\xb,\yb)} L\not=\emptyset.\]
Hence, $\bar n:=\dim \bar L\geq n$ and if $\bar n=n$ then $\Sp F(\xb,\yb)=\{\bar L\}$ follows. Now assume that $\bar n>n$ and choose an orthogonal basis $w_i$, $i=1,\ldots,\bar n$, of $\bar L$. Since $\zb:=(\xb,\yb)\in(\cl \OO_F)\setminus \OO_F$, there must be a sequence $z_k:=(x_k,y_k)\longsetto{\OO_F}(\xb,\yb)$ satisfying $z_k\not =\zb$ $\forall k$. Taking into account \eqref{EqTangStrictProto}, for every $k$ sufficiently large and every $i=1,\ldots,\bar n$, we can find elements $w_i^k\in L_k:=T_{\gph F}(z_k)$ with $\norm{w_i^k-w_i}<1/\sqrt{\bar n}$. We claim now that for those $k$ the vectors $w_i^k$, $i=1,\ldots,\bar n$, are linearly independent. Assume on the contrary, that for some large $k$ these vectors are linearly dependent and we can find reals $\lambda_i^k$, $i=1,\ldots,\bar n$, with
$\sum_{i=1}^{\bar n}\lambda_i^k w_i^k=0$ and $\sum_{i=1}^{\bar n}(\lambda_i^k)^2=1.$
Then by orthogonality of $w_i$, $i=1,\ldots,\bar n$, together with the Cauchy-Schwarz inequality  we obtain that
\begin{align*}1&=\norm{\sum_{i=1}^{\bar n}\lambda_i^kw_i}\leq \norm{\sum_{i=1}^{\bar n}\lambda_i^kw_i^k}+\norm{\sum_{i=1}^{\bar n}\lambda_i^k(w_i-w_i^k)}\\
&\leq\Big(\sum_{i=1}^{\bar n}(\lambda_i^k)^2\Big)^{1/2}\Big(\sum_{i=1}^{\bar n}\norm{w_i-w_i^k}^2\Big)^{1/2}<1\cdot(\sum_{i=1}^{\bar n}\frac 1{\bar n})^{1/2}=1,
\end{align*}
a contradiction. Thus $w_i^k$, $i=1,\ldots,\bar n$, are linearly independent and $\dim L_k\geq \bar n$ follows. However, by definition of the set $\OO_F$ we have $\dim L_k=n$, a contradiction. Hence $\bar n=n$ and $\Sp F(\xb,\yb)=\{\bar L\}$ holds true.

To show the second assertion, assume that $F$ is graphically Lipschitzian of dimension $n$ at $(\xb,\yb)$ and let $W$, $\Phi$, $U$ and $f$ be as in Definition \ref{DefGraphLip}. Note that for all sequences $t_k\downarrow0$, $(x_k,y_k)\longsetto{\gph F\cap W}(\xb,\yb)$ and $(\hat x_k,\hat y_k)\longsetto{\gph F\cap W}(\xb,\yb)$ such that
$\lim_{k\to\infty}{((\hat x_k,\hat y_k)-(x_k,y_k))}/{t_k}=:(\xi,\eta)$
exists, we have that $\Phi(x_k,y_k)=:(u_k,f(u_k))\longsetto{\gph f}(\bar u,f(\bar u))=\Phi(\xb,\yb)$, $\Phi(\hat x_k,\hat y_k)=:(\hat u_k,f(\hat u_k))\longsetto{\gph f}(\bar u,f(\bar u))$ and
\begin{align*}
  \lim_{k\to\infty}\frac{(\hat u_k,f(\hat u_k))-(u_k, f(u_k))}{t_k}&=\lim_{k\to\infty}\frac{\nabla \Phi(\xb,\yb)\big((\hat x_k,\hat y_k)-(x_k,y_k)\big)+\oo(\norm{(\hat x_k,\hat y_k)-(x_k,y_k)})}{t_k}\\
  &=\nabla \Phi(\xb,\yb)(\xi,\eta).
\end{align*}
Analogously, for all sequences $t_k\downarrow 0$, $u_k\longsetto{U}\bar u$, $\hat u_k\longsetto{U}\bar u$ such that
\[\lim_{k\to\infty}\frac{(\hat u_k,f(\hat u_k))-(u_k,f(u_k))}{t_k}=:(\sigma,\tau)\]
exists, we have that $\Phi^{-1}(u_k,f(u_k))=:(x_k,y_k)\longsetto{\gph F\cap W}(\xb,\yb)$, $\Phi^{-1}(\hat u_k,f(\hat u_k))=:(\hat x_k,\hat y_k)\longsetto{\gph F\cap W}(\xb,\yb)$ and
\[\lim_{k\to\infty}\frac{(\hat x_k,\hat y_k)-(x_k,y_k)}{t_k}=\nabla\Phi(\xb,\yb)^{-1}(\sigma,\tau).\]
Thus, $F$ is strictly proto-differentiable at $(\xb,\yb)$ if and only if $f$ is strictly proto-differentiable at $(\bar u,f(\bar u))$, which is the same as $f$ being strictly differentiable at $\bar u$. Now, by Lemma \ref{LemSCDGraphLipsch} the SC derivative $\Sp F(\xb,\yb)$ is a singleton  if and only if $\overline \nabla f(\bar u)$ is a singleton, which is equivalent to strict differentiability of $f$ at $\bar u$. Combining these statements it follows that $F$ is strictly proto-differentiable at $(\xb,\yb)$ if and only if $\Sp F(\xb,\yb)$ is a singleton.
\end{proof}
For graphically Lipschitzian mappings, the SCD-ss$^*$ property is closely related with the G-semi\-smooth\-ness.
\begin{proposition}\label{PropSCDss-ss}
  Assume that the mapping $F:\R^n\tto\R^m$ is graphically Lipschitzian of dimension $n$ at $(\xb,\yb)\in\gph F$ and let $\Phi$ and $f$ be as in Definition \ref{DefGraphLip}. Then $F$ is SCD-ss$^*$ at $(\xb,\yb)$ if and only if $f$ is $\onabla f$-semismooth at $\bar u$, where $(\bar u,f(\bar u))=\Phi(\bar x,\bar y)$.
\end{proposition}
\begin{proof}
Consider $z:=(x,y)\in\gph F\cap W$ and the corresponding pair $(u,f(u))=\Phi(z)$, where $W$ is as in Definition \ref{DefGraphLip}. Let $L\in\Sp F(z)$ and $A\in\onabla f(u)$ such that $L=\nabla\Phi(z)^{-1}(\rge(I,A))$ according to Lemma \ref{LemSCDGraphLipsch}. Then, setting $\zb:=(\xb,\yb)$, we obtain that
\begin{align*}\dist{z-\zb,L}&=\inf_{p\in\R^n}\norm{\nabla \Phi(z)^{-1}\big(\nabla\Phi(z)(z-\zb)-(p,Ap)\big)}\\
&=\inf_{p\in\R^n}\norm{\nabla \Phi(z)^{-1}\big(\Phi(z)-\Phi(\zb)+\oo(\norm{z-\zb})-(p,Ap)\big)}\\
&\leq\norm{\nabla \Phi(z)^{-1}}\big(\inf_{p\in\R^n}\norm{(u-\bar u-p,f(u)-f(\bar u)-Ap)}+\oo(\norm{z-\zb})\big)\\
&\leq\norm{\nabla \Phi(z)^{-1}}\big(\norm{f(u)-f(\bar u)-A(u-\bar u)}+\oo(\norm{z-\zb})\big).
\end{align*}
On the other hand there holds
\begin{align*}
  &\inf_{p\in\R^n}\norm{(u-\bar u-p,f(u)-f(\bar u)-Ap)}^2=\inf_{q\in\R^n}\norm{q}^2+\norm{f(u)-f(\bar u)-A(u-\bar u)-Aq}^2\\
  &\geq  \inf_{w\in\R^n}\frac{\norm{w}^2}{\norm{A}^2}+\norm{f(u)-f(\bar u)-A(u-\bar u)-w}^2=\frac1{1+\norm{A}^2}\norm{f(u)-f(\bar u)-A(u-\bar u)}^2,
\end{align*}
resulting in
\begin{align*}\dist{z-\zb,L}&=\inf_{p\in\R^n}\norm{\nabla \Phi(z)^{-1}\big(\Phi(z)-\Phi(\zb)+\oo(\norm{z-\zb})-(p,Ap)\big)}\\
&\geq\inf_{p\in\R^n}\norm{\nabla \Phi(z)^{-1}(u-\bar u-p,f(u)-f(\bar u)-Ap)}-\norm{\nabla \Phi(z)^{-1}}\oo(\norm{z-\zb})\\
&\geq\frac 1{\norm{\nabla\Phi(z)}}\inf_{p\in\R^n}\norm{(u-\bar u-p,f(u)-f(\bar u)-Ap)}-\norm{\nabla \Phi(z)^{-1}}\oo(\norm{z-\zb})\\
&\geq \frac 1{\norm{\nabla\Phi(z)}\sqrt{1+\norm{A}^2}}\norm{f(u)-f(\bar u)-A(u-\bar u)}-\norm{\nabla \Phi(z)^{-1}}\oo(\norm{z-\zb}).
\end{align*}
Since $\norm{A}$ is bounded by the Lipschitz constant $\kappa$ of $f$ and
\[\norm{u-\bar u}\begin{cases}\geq \frac{\norm{(u-\bar u,f(u)-f(\bar u))}}{\sqrt{1+\kappa^2}} \geq\frac{\norm{z-\zb}}{\norm{\nabla\Phi(z)^{-1}}\sqrt{1+\kappa^2}}+\oo(\norm{z-\zb})&,\\
\leq\norm{(u-\bar u,f(u)-f(\bar u))}\leq\norm{\nabla\Phi(z)}\norm{z-\zb}+\oo(\norm{z-\zb}),\end{cases}\]
 we conclude from these inequalities that
\[\Limsup_{z\longsetto{\gph F}\zb}\sup_{L\in\Sp F(z)}\frac{\dist{z-\zb,L}}{\norm{z-\zb}}=0\ \Leftrightarrow\
\limsup_{u\to\bar u}\sup_{A\in\onabla f(u)}\frac{\norm{f(u)-f(\bar u)-A(u-\bar u)}}{\norm{u-\bar u}}=0
\]
and the assertion follows from Proposition \ref{PropSCD_ss_Primal} and the definition of $\overline \nabla f$-semismoothness.
\end{proof}

\begin{corollary}
  Assume that the mapping $F:\R^n\tto\R^m$ is graphically Lipschitzian of dimension $n$ at $(\xb,\yb)\in\gph F$ and  SCD-ss${^*}$ around $(\xb,\yb)$. Then there is an open neighborhood $W$ of $(\xb,\yb)$ such that for almost all $(x,y)\in\gph F\cap W$ (with respect to the $n$-dimensional Hausdorff measure) the mapping $F$ is strictly proto-differentiable at $(x,y)$ and $\Sp F(x,y)$ is a singleton.
\end{corollary}
\begin{proof}
 According to the notation of Definition \ref{DefGraphLip}, let $\Omega_F$ denote the set of all points in $\gph F\cap W$ where  $F$ is strictly proto-differentiable and let $\Omega_f$ denote the set of all points $u\in U\subset\R^n$ where $f$ is strictly differentiable. We may also assume that $F$ is SCD-ss$^*$ on $\gph F\cap W$ and therefore $f$ is $\onabla f$-ss on $U$ by Proposition \ref{PropSCDss-ss}. We conclude from Theorem \ref{ThNorkinR_m} that the set $U\setminus \Omega_f$ has Lebesgue measure zero and the assertion follows from the relation
  \[\Omega_F=\Phi^{-1}(\{(u,f(u))\mv u\in\Omega_f\}),\]
  which holds true by Theorem \ref{ThStrictProtoDiff} and Lemmas \ref{LemStrictDiff}, \ref{LemSCDGraphLipsch}.
\end{proof}
\if{
\begin{example}Consider a maximally monotone mapping $F:\R^n\tto\R^n$. Then $F$ is graphically Lipschitzian with transformation mapping $\Phi(x,y)=(x+y,x)$ and $f=(I+F)^{-1}$, the resolvent. By Lemma \ref{LemSCDGraphLipsch} there holds
\[\Sp F(x,y)=\{\rge(I-A,A)\mv A\in \onabla (I+F)^{-1}(x+y)\}.\]
\end{example}
}\fi

\section{On the semismooth$^*$ property of solution maps to inclusions}
In this section we return to the motivating problem \eqref{EqCompProbl}.
\begin{theorem}\label{ThSolMapSSGen}
  Consider a mapping $F:\R^n\times\R^m\tto\R^l$ and let
  \begin{equation}\label{EqSolMap}
    S(x):=\{y\in \R^m\mv 0\in F(x,y)\}
  \end{equation}
  be the solution map of the inclusion $0\in F(x,y)$. Assume that we are given $(\xb,\yb)\in\gph S$ such that the qualification condition
  \begin{equation}\label{EqQualCond}(0,0)\in D^*F(\xb,\yb,0)(z^*)\ \Rightarrow\ z^*=0\end{equation}
  holds. Then
   there is a neighborhood $U$ of $(\xb,\yb)$ such that for all $(x,y)\in\gph S\cap U$ one has
  \begin{equation}\label{EqInclCoderiv}
        D^*S(x,y)(y^*)\subset \Sigma(x,y)(y^*):=\{x^*\mv \exists z^*\in\R^l\mbox{ with } (x^*,-y^*)\in D^*F(x,y,0)(z^*)\},\ y^*\in\R^m.
  \end{equation}
  Moreover, $S$ is semismooth$^*$ at $(x,y)$ with respect to the mapping $\Gamma:\R^n\times\R^m\tto\R^m\times\R^n$ given by
  \begin{equation*}
  \Gamma(u,v)=\{(y^*,x^*)\mv x^*\in\Sigma(u,v)(y^*)\},
  \end{equation*}
  provided $F$ is semismooth$^*$ at $(x,y,0)$.
\end{theorem}
\begin{proof}
  We claim that there is some open neighborhood $U$ of $(\xb,\yb)$ and some real $\gamma>0$ such that for every pair $(x,y)\in\gph S\cap U$, every $z^*\in \R^l$, $\norm{z^*}=1$ and every $(x^*,y^*)\in D^*F(x,y,0)(z^*)$ one has $\norm{(x^*,y^*)}\geq 1/\gamma$. Assume on the contrary that there are sequences $(x_k,y_k)\longsetto{\gph S}(\xb,\yb)$ and $(z_k^*,x_k^*,y_k^*)\in\gph D^*F(x_k,y_k,0)$ with $\norm{z_k^*}=1$ and $\norm{(x_k^*,y_k^*)}\to 0$ as $k$ tends to $\infty$. By possibly passing to a subsequence we can assume that $z_k^*$ converges to some $z^*$ with $\norm{z^*}=1$ and by definition of the limiting coderivative we obtain $(0,0,z^*)\in\gph D^*F(\xb,\yb,0)$ contradicting \eqref{EqQualCond}. Hence our claim holds true.

  Then we also have for all $(x,y)\in\gph S\cap U$ that
  \[(0,0)\in D^*F(x,y,0)(z^*)\ \Rightarrow\ z^*=0\]
  and the inclusion \eqref{EqInclCoderiv} follows from \cite[Theorem 4.46]{Mo06a}.

   Now assume that $F$ is semismooth$^*$ at $(x,y,0)$ with $(x,y)\in U$. Fix $\epsilon>0$ and choose $\delta>0$ such that $\B_\delta(x,y)\subset U$ and for all $(x',y')\in \gph S\cap \B_\delta(x,y)$ and all $(z^*,x^*,y^*)\in \gph D^*F(x',y',0)$ one has
  \begin{equation}\label{EqAuxSS1}\vert \skalp{z^*,0-0}-\skalp{x^*,x'-x}-\skalp{y^*,y'-y}\vert\leq \epsilon\norm{(x',y',0)-(x,y,0)}\norm{(z^*,x^*,y^*)}.\end{equation}
  Next consider $(y^*,x^*)\in\Gamma(x',y')=\gph\Sigma(x',y')$ and the corresponding $z^*$ verifying $(z^*,x^*,-y^*)\in \gph D^*F(x',y',0)$. By our claim we have $\norm{z^*}\leq \gamma\norm{(x^*,-y^*)}=\gamma\norm{(x^*,y^*)}$ and from \eqref{EqAuxSS1} we obtain
  \[\vert \skalp{y^*,y'-y}-\skalp{x^*,x'-x}\vert\leq \epsilon\sqrt{1+\gamma^2}\norm{(x^*,y^*)}\norm{(x',y')-(x,y)}.\]
  Hence $S$ is $\Gamma$-ss$^*$ at $(x,y)$ and the proof  is complete.
\end{proof}
\begin{remark}
This result extends slightly \cite[Corollary 4.47] {Mo06a}, because estimate (\ref {EqInclCoderiv}) holds true not only at the reference point $(\xb,\yb)$ but also on a neighborhood $U$ of it.
\end{remark}
\begin{theorem}\label{ThSolMapSSLip}
In the setting of Theorem \ref{ThSolMapSSGen}, assume that there is an open neighborhood $\Omega\times V$ of $(\xb,\yb)\in\gph S$ such that $\gph S\cap(\Omega\times V)$ is the graph of a  continuous mapping $\sigma:\Omega\to\R^m$. Further assume that the strengthened qualification condition
\begin{equation}\label{EqStrongQual}
(x^*,0)\in D^*F(\xb,\yb,0)(z^*)\ \Rightarrow z^*=0,\ x^*=0
\end{equation}
holds. Then there is some open neighborhood $U$ of $\xb$ such that $\sigma$ is Lipschitz continuous on $U$. If, in addition, $F$ is semismooth$^*$ at $(x,\sigma(x),0)$ for some  $x\in U$, then for every function $\varphi:\R^n\times\R^m\to\R$, which is semismooth at $(x,\sigma(x))$ with respect to some mapping $\Phi:\R^n\times\R^m\tto\R^{1\times(n+m)}$, the mapping $\theta:U\to\R$, $\theta(u):=\varphi(u,\sigma(u))$, is semismooth at $x$ with respect to the mapping $\Theta:\R^n\tto\R^{1\times n}$ given by
\[\Theta(u):=\{g_x^T+{x^*}^T\mv x^*\in \Sigma(u,\sigma(u))(g_y), (g_x^T,g_y^T)\in \Phi(u,\sigma(u))\},\  u\in \dom\Theta:=U,\]
where $\Sigma$ is given by \eqref{EqInclCoderiv}.
\end{theorem}
\begin{proof}
  We claim that there is some neighborhood $U\subset\Omega$ of $\xb$ and some real $\gamma>0$ such that for all $x\in U$ and all $(z^*,x^*,y^*)\in \gph D^*F(x,\sigma(x),0)$ one has
  $\norm{y^*}\geq \norm{(x^*,z^*)}/\gamma.$
  Assume on the contrary that there are sequences $x_k\to\xb$ and $(z_k^*,x_k^*,y_k^*)\in \gph D^*F(x_k,\sigma(x_k),0)$ with $\norm{(x^*_k,z^*_k)}=1$ and $y_k^*\to 0$ as $k\to\infty$. By possibly passing to a subsequence we can assume that $(x_k^*,z_k^*)$ converges to some $(x^*,z^*)$ with $\norm{(x^*,z^*)}=1$. It follows that $(z^*,x^*,0)\in \gph D^*F(\xb,\sigma(\xb),0)$ contradicting  \eqref{EqStrongQual}. Hence our claim holds true. Consequently, there holds
\begin{equation*}
(x^*,0)\in D^*F(x,\sigma(x),0)(z^*)\ \Rightarrow z^*=0,\ x^*=0 \end{equation*}
for every $x\in U$.
By possibly shrinking $U$ we may assume that $U$ is an open, bounded and convex neighborhood of $\xb$. By invoking \cite[Corollary 4.60]{Mo06a} we obtain that $\sigma$ is Lipschitz near every $x\in U$ with
\[\lip \sigma(x)\leq \sup\{\norm{x^*}\mv \exists z^* \mbox{ with }(z^*,x^*,-y^*)\in \gph D^*F(x,\sigma(x),0),\norm{y^*}\leq 1\}\leq \gamma,\ x\in U.\]
Thus, by \cite[Theorem 9.2]{RoWe98}, we obtain that $\sigma$ is Lipschitz continuous on $U$ with constant $\gamma$. Hence, by virtue of \eqref{EqInclCoderiv} we have
\begin{equation}\label{EqNonemptySigma}\emptyset\not=D^*\sigma(x)(y^*)\subset\Sigma(x,\sigma(x),0)\ \forall x\in U\ \forall y^*\in\R^m.\end{equation}

Next consider $x\in U$ and assume that $F$ is semismooth$^*$ at $(x,\sigma(x),0)$ and consider any function $\varphi:\R^n\to\R$ which is semismooth at $x$ with respect to some  mapping $\Phi:\R^n\times\R^m\tto\R^{1\times(n+m)}$. Then $\Phi$ is locally bounded at $(x,\sigma(x))$ and we can find some open convex neighborhood $\tilde U\subset U$ of $x$ such that
$\{(u,\sigma(u))\mv u\in\tilde U\} \subset\dom\Phi$ and
\[ M:=\sup\{\norm{(g_x^T,g_y^T)}\mv (g_x^T,g_y^T)\in \Phi( u,\sigma(u)), u\in \tilde U\}<\infty.\]
Further, for every $u\in U$ and every $y^*\in\R^m$ there holds
\[\sup\{\norm{x^*}\mv x^*\in \Sigma(u,\sigma(u))(y^*)\}=\sup\{\norm{x^*}\mv \exists z^*\mbox{ with }(z^*,x^*,-y^*)\in D^*F(u,\sigma(u),0)\}\leq \gamma\norm{y^*}.\]
Hence, for every $u\in \tilde U$ and every ${u^*}^T\in \Theta(u)$ we have
\[\norm{u^*}\leq \sup\{\norm{g_x}+\norm{x^*}\mv x^*\in \Sigma(u,\sigma(u))(g_y),\ (g_x^T,g_y^T)\in \Phi(u,\sigma(u))\}\leq M+\gamma M,\]
verifying the local boundedness of $\Theta$ at $x$. Further, since $\dom\Theta$ is a neighborhood of $(x,\sigma(x))$ and taking into account \eqref{EqNonemptySigma}, we obtain that $\dom\Theta$ is a neighborhood of $x$. Finally, let $\epsilon>0$ be given and consider $\delta>0$ such that for all $x'\in U$ with $(x',\sigma(x'),0)\in \B_{\delta}(x,\sigma(x),0)$ and all $(z^*,x^*,y^*)\in\gph D^*F(x',\sigma(x'),0)$ one has
\[\vert\skalp{z^*,0-0}-\skalp{x^*,x'-x}-\skalp{y^*,\sigma(x')-\sigma(x)}\vert\leq \epsilon\norm{(z^*,x^*,y^*)}\norm{(x'-x,\sigma(x')-\sigma(x),0-0)}\]
due to the semismoothness$^*$ property of $F$ at $(x,\sigma(x),0)$. By possibly decreasing $\delta$ we may also assume that for all $x'\in \tilde U$ with $(x',\sigma(x'))\in \B_\delta(x,\sigma(x))$ and all $(g_x^T,g_y^T)\in \Phi(x',\sigma(x'))$ there holds
\[\vert \varphi(x',\sigma(x'))- \varphi(x,\sigma(x))-\skalp{g_x,x'-x}-\skalp{g_y,\sigma(x')-\sigma(x)}\vert\leq\epsilon\norm{(x'-x,\sigma(x')-\sigma(x))}.\]
Choose $\delta'\in (0,\delta/\sqrt{1+\gamma^2})$ such that $\B_{\delta'}(x)\subset U$ and consider $x'\in \B_{\delta'}(x)$ and ${u^*}^T\in \Theta(x')$ together with $(g_x^T,g_y^T)\in\Phi(x',\sigma(x'))$ and $x^*\in \Sigma(x',\sigma(x'))(g_y)$ such that $u^*=g_x^T+{x^*}^T$. Further choose $z^*$ such that $(z^*,x^*,-g_y)\in \gph D^*F(x',\sigma(x'),0)$ according to the definition of $\Sigma$. Then $(x',\sigma(x'))\in\B_\delta(x,\sigma(x))$ and consequently
\begin{align*}
  &\lefteqn{\vert \varphi(x',\sigma(x'))-\varphi(x,\sigma(x))-\skalp{u^*,x'-x}\vert}\\
  &\leq\vert \varphi(x',\sigma(x'))-\varphi(x,\sigma(x))-\skalp{g_x,x'-x}-\skalp{g_y,\sigma(x')-\sigma(x)}\vert+\vert \skalp{g_y,\sigma(x')-\sigma(x)}-\skalp{x^*,x'-x}\vert\\
  &\leq \epsilon\norm{(x'-x,\sigma(x')-\sigma(x))}+\epsilon\norm{(z^*,x^*,-g_y)}\norm{(x'-x,\sigma(x')-\sigma(x),0-0)}\\
  &\leq\epsilon\norm{x'-x}\sqrt{1+\gamma^2}(1+\sqrt{1+\gamma^2}\norm{g_y})\leq \epsilon\norm{x'-x}\sqrt{1+\gamma^2}(1+\sqrt{1+\gamma^2}M),
\end{align*}
verifying that $\theta$ is semismooth with respect to $\Theta$ at $x$. This completes the proof.
\end{proof}
\begin{remark}\label{RemConstrPsi}
In the setting of Theorem \ref{ThSolMapSSLip}, consider a point $x\in U$ where $F$ is semismooth$^*$ at $(x,\sigma(x),0)$ and consider the linear functions $\beta_i:\R^n\times\R^m\to\R$, $\beta_i(x,y)=y_i$, $i=1,\ldots,m$, together with mappings $\Psi_i: U\tto\R^{1\times n}$ given by  $\Psi_i(u):=\{{u^*}^T\mv u^*\in\Sigma(u,\sigma(u))(-e_i)\}$, where $e_i$ denotes the $i$-th unit vector in $\R^m$. It follows from Theorem \ref{ThSolMapSSLip} that $\sigma_i(u)=\beta_i(u,\sigma(u))$ is semismooth at $x$ with respect to $\Psi_i$, $i=1,\ldots,m$. Hence,  $\sigma$ is also semismooth with respect to the mapping $\Psi$, which assigns to each $u\in U$ the set of matrices whose $i$-th row is an element from $\Psi_i(u)$ as in \eqref{EqPsi_from_Components}.
\end{remark}

\begin{corollary}\label{CorSS*Imp}
  In the setting of Theorem \ref{ThSolMapSSGen}, assume that we are given an open set $\Omega$ and a continuous mapping $\sigma:\Omega\to\R^m$ such that for every $x\in \Omega$ the mapping $F$ is semismooth$^*$ at $(x,\sigma(x),0)$, the qualification condition \eqref{EqQualCond}
  \if{
    \begin{equation*}
        (x^*,0)\in D^*F(x,\sigma(x),0)(z^*)\ \Rightarrow z^*=0, x^*=0
    \end{equation*}
    }\fi
    holds and there exists a neighborhood $U\times V$ of $(x,\sigma(x))$ such that $\gph S\cap(U\times V)=\gph\sigma\cap(U\times V)$.\\
    Then $\sigma$ is semismooth on $\Omega$ with respect to some mapping $\Psi:\R^n\tto\R^{n\times m}$, which can be constructed analogously to Remark \ref{RemConstrPsi}. Further, for every function $\varphi:D\to \R$, where $D$ is an open subset of $\R^n\times\R^m$ containing $\gph \sigma$, which is semismooth on $\gph\sigma$ with respect to some mapping $\Phi:\R^n\times\R^m\tto\R^{1\times(n+m)}$, there holds that the mapping $\theta:\Omega\to\R$ given by $\theta(x)=\varphi(x,\sigma(x))$ is semismooth on $\Omega$ with respect to the mapping $\Theta:\R^n\tto\R^{1\times n}$ given by
        \[\Theta(x):=\{g_x^T+{x^*}^T\mv x^*\in \Sigma(x,\sigma(x))(g_y), (g_x^T,g_y^T)\in \Phi(x,\sigma(x))\},\ x\in \dom\Theta:=\Omega.\]
\if{
    Then the following statements hold:
    \begin{enumerate}
    \item $\sigma$ is locally Lipschitz on $\Omega$.
    \item $\sigma$ is semismooth on $\Omega$ with respect to some mapping $\tilde\Psi:\R^n\tto\R^{n\times m}$, which can be constructed analogously to Remark \ref{RemConstrPsi}. In particular, $\sigma$ is strictly  differentiable almost everywhere in $\Omega$.
    \item Consider a function $\varphi:D\to \R$, where $D$ is an open subset of $\R^n\times\R^m$ containing $\gph \sigma$ and assume that  $\varphi$ is semismooth on $\gph\sigma$ with respect to some mapping $\Phi:\R^n\times\R^m\tto\R^{1\times(n+m)}$. Then the mapping $\theta:\Omega\to\R$ given by $\theta(x)=\varphi(x,\sigma(x))$ is semismooth on $\Omega$ with respect to the mapping $\Psi:\R^n\tto\R^{1\times n}$ given by
        \[\Psi(x):=\{g_x^T+{x^*}^T\mv x^*\in \Sigma(x,\sigma(x))(g_y), (g_x^T,g_y^T)\in \Phi(x,\sigma(x))\},\ x\in \dom\Psi:=\Omega.\]
  \end{enumerate}
  }\fi
\end{corollary}
This corollary enables us to apply  the ImP approach to the numerical solution of a class of problems (\ref{EqCompProbl}) with nonsmooth objectives $\varphi$ via a suitable bundle method. Apart from the case when $\sigma$ is a single-valued and Lipschitz localization of $S$ for all $x\in\Omega$, however, the construction of $\Psi$ according to Remark \ref{RemConstrPsi} is rather complicated. We will now show that the theory considerably simplifies  in the case of SCD mappings.

For an SCD mapping $F:\R^n\times\R^m\tto\R^m$, the subspaces $L^*\in \Sp^*F(x,y,z)$ belong to $\Z_{m,n+m}$ and therefore $L^*\subset\R^{m+n+m}$. By identifying $\R^{m+n+m}$ with $\R^m\times\R^n\times\R^m$, we denote the elements of $L^*$ by $(z^*,x^*,y^*)\in\R^m\times\R^n\times\R^m$.
We start our analysis with the following auxiliary result:
\begin{proposition}\label{PropRegSubsp}
 For every subspace $L^*\in\Z_{m,{n+m}}$ satisfying the implication
 \begin{equation}\label{EqRegSubspace}(z^*,x^*,0)\in L^*\Rightarrow z^*=0,\ x^*=0\end{equation}
  there are unique  matrices $Z_{L^*}\in\R^{m\times m}$ and $X_{L^*}\in\R^{n\times m}$ such that
\begin{equation}\label{EqDualBasRepr}L^*=\rge(Z_{L^*}, X_{L^*},I)=\{(Z_{L^*}p, X_{L^*}p,p)\mv p\in\R^m\}.\end{equation}
In addition, there holds
\[\kappa(L^*):=\bnorm{\myvec{Z_L^*\\X_{L^*}}}=\sup\{\norm{(z^*,x^*)}\mv (z^*,x^*,y^*)\in L^*, \norm{y^*}\leq 1\}.\]
\end{proposition}
\begin{proof}
  Consider an $(m+n+m)\times m$ matrix $B$ whose columns form a basis for the $m$-dimensional subspace $L^*\subset R^{m+n+m}$. We can partition $B$ into the matrices $B_{z^*}\in\R^{m\times m}$, $B_{x^*}\in\R^{n\times m}$ and $B_{y^*}\in\R^{m\times m}$ so that
  \begin{equation*}
  B=\myvec{B_{z^*}\\B_{x^*}\\B_{y^*}}.\end{equation*}
  We claim that condition \eqref{EqRegSubspace} ensures the nonsingularity of  the matrix $B_{y^*}$. Assume on the contrary that there is some $0\not=p\in\R^m$ satisfying $B_{y^*}p=0$. Then we have
  $Bp=(B_{z^*}p,B_{x^*}p,0)\in L^*$ and consequently $B_{z^*}p=0$, $B_{x^*}p=0$ due to \eqref{EqRegSubspace}. Thus $Bp=0$, which is not possible since the columns of $B$ form a basis for $L^*$ and are therefore linearly independent. Hence the matrix $B_{y^*}$ is nonsingular and the columns of the matrix $B B_{y^*}^{-1}$ form another basis for $L^*$. Representation \eqref{EqDualBasRepr} now follows with
  $Z_{L^*}=B_{z^*}B_{y^*}^{-1},\ X_{L^*}=B_{x^*}B_{y^*}^{-1}$
  and it is easy to see that  it is unique. Finally,
  \[\sup\{\norm{(z^*,x^*)}\mv (z^*,x^*,y^*)\in L^*, \norm{y^*}\leq 1\}= \sup\{\norm{(Z_{L^*}y^*,X_{L^*}y^*)}\mv \norm{y^*}\leq 1\}=\bnorm{\myvec{Z_L^*\\X_{L^*}}}\]
  and we are done.
\end{proof}

Let us  denote by $\Z_{m,n+m}^{\rm reg}$ the collection of all subspaces $L^*\in\Z_{m,n+m}$ satisfying \eqref{EqRegSubspace}.
In the following statements we collect some useful relations concerning the subspaces $L^*$ from $\Z_{m,n+m}^{\reg}$ and the moduli $\kappa (L^*)$ and ${\rm scd\, reg\,}F(x,y,z)$ (introduced below).
\begin{lemma}
  Assume that $F:\R^n\times\R^m\tto\R^m$ has the SCD property at $(x,y,z)\in\gph F$. Then the following statements are equivalent:
  \begin{enumerate}
    \item[(i)]$(z^*,x^*,0)\in(\bigcup\Sp^*F)(x,y,z) \Rightarrow z^*=0,\ x^*=0.$
    \item[(ii)] $\Sp^*F (x,y,z)\subset \Z_{m,n+m}^{\rm reg}$.
    \item[(iii)]${\rm scd\, reg\,}F(x,y,z):=\sup\{\norm{(z^*,x^*)}\mv (z^*,x^*,y^*)\in (\bigcup\Sp^*F)(x,y,z),\norm{y^*}\leq 1\}<\infty$.
  \end{enumerate}
  In this case one also has
  \begin{equation}\label{EqSCDreg1} {\rm scd\, reg\,}F(x,y,z):=\sup\{\kappa(L^*)\mv L^*\in \Sp^*F(x,y,z)\}.\end{equation}
\end{lemma}
\begin{proof}
The implication (i)$\Rightarrow$(ii) follows immediately from Proposition \ref{PropRegSubsp}. We prove the implication (ii)$\Rightarrow$(iii). Assuming that (iii) does not hold there are sequences $(z_k^*,x_k^*,y_k^*)\in L_k^*\subset (\bigcup\Sp^*F)(x,y,z)$ with $\norm{y_k^*}\leq 1$ for all $k$ and $\norm{(z_k^*,x_k^*)}\to\infty$ as $k\to\infty$. Since $L^*_k$ is a subspace, we have $(\tilde z^*_k,\tilde x^*_k,\tilde y^*_k):= (z_k^*,x_k^*,y_k^*)/\norm{(z_k^*,x_k^*,y_k^*)}\in L_k^*$ for all $k$. Taking into account that the metric space $\Z_{m,n+m}$ is compact, by possibly passing to a subsequence we can assume that $L_k^*$ converges to some $L^*\in \Z_{m,n+m}$ and $(\tilde z^*_k,\tilde x^*_k, \tilde y^*_k)$ converges to some $(\tilde z^*,\tilde x^*,\tilde y^*)\in L^*$ with $\norm{(\tilde z^*,\tilde x^*,\tilde y^*)}=1$. Since $\Sp^*F(x,y,z)$ is closed in $\Z_{m,n+m}$, there holds $L^*\in \Sp^*F(x,y,z)$ and because of $\tilde y_k^*/\norm{(z_k^*,x_k^*,y_k^*)}\to 0$ we have $\tilde y^*=0$. Hence $(\tilde z^*,\tilde x^*,0)\in L^*\in \Sp^*F(x,y,z)$ and $\norm{(\tilde z^*,\tilde x^*)}=1$ showing $L^*\not\in \Z_{m,n+m}^{\rm reg}$. This proves the implication (ii)$\Rightarrow$(iii). The implication (iii)$\Rightarrow$(i)  holds trivially. Finally, formula \eqref{EqSCDreg1} is a consequence of the definitions of ${\rm scd\,reg}F$, $(\bigcup\Sp^*F)$ and $\kappa(L^*)$.
\end{proof}
By following the proof of \cite[Proposition 4.8]{GfrOut22a} one can easily show the following result.
\begin{proposition}\label{PropSCDRegOSC}
Assume that $F:\R^n\times\R^m\tto\R^m$ has the SCD property around some point $(x,y,z)\in\gph F$. Then
\begin{equation*}
\limsup_{(x',y',z')\longsetto{\gph F}(x,y,z)}{\rm scd\, reg\,}F(x',y',z')\leq {\rm scd\, reg\,}F(x,y,z).\end{equation*}
\end{proposition}
\if{
\begin{proof}
  If $\gamma:={\rm scd\, reg\,}F(x,y,z)=\infty$ there is nothing to prove.  Thus we can suppose that  $\gamma< \infty$ and consequently Lemma \ref{LemEquivReg}(i) is fulfilled by the equivalences stated in this lemma. Assume now on the contrary that the inequality \eqref{EqOSCReg} does not hold  and we can find  sequences $(x_k,y_k,z_k)\to (x,y,z)$ and $(z_k^*,x_k^*,y_k^*)\in L_k^*\in\Sp^*F(x_k,y_k,z_k)$ with $\norm{y_k}^*\leq 1$ and $\limsup_{k\to\infty}\norm{(z^*_k,x^*_k)}=:\bar\gamma>\gamma$. If $\bar\gamma$ is finite, by possibly passing to a subsequence we can assume that $(z_k^*,x_k^*,y_k^*)$ converges to some $(\tilde z^*,\tilde x^*,\tilde y^*)$ with $\bar\gamma=\norm{(\tilde z^*,\tilde x^*)}$.  By possibly passing to a subsequence once more we can assume that $L_k^*$ converges to some $L^*\in \Sp^*F(x,y,z)$. Since $(\tilde z^*,\tilde x^*,\tilde y^*)\in L^*$ and $\norm{\tilde y^*}\leq 1$, we obtain the contradiction $\gamma<\bar\gamma=\norm{(\tilde z^*,\tilde x^*)}\leq \kappa(L^*)\leq\gamma$. If $\bar\gamma=\infty$ we can employ the same arguments already used in the proof of Lemma \ref{LemEquivReg} to obtain a contradiction.
\end{proof}
}\fi
The next statement pertains to the $\Psi$-ss property of a single-valued continuous selection of \eqref{EqSolMap} in the SCD case.
\begin{theorem}
  Consider a mapping $F:\R^n\times\R^m\tto\R^m$ and  a point $(\xb,\yb)$ with $0\in F(\xb,\yb)$.  Assume that $F$ has the SCD property around $(\xb,\yb,0)$ and ${\rm scd\,reg\,}F(\xb,\yb,0)<\infty$. Further suppose that there exists a  neighborhood $\Omega$ of $\xb$ and a continuous function $\sigma:\Omega\to\R^m$ with $\sigma(\xb)=\yb$ and $0\in F(x,\sigma(x))$, $x\in\Omega$. Then there exists an open neighborhood $U\subset \Omega$ of $\xb$ such that for every $x\in U$ there holds ${\rm scd\,reg\,}F(x,\sigma(x),0)<\infty$ and $\sigma$ is semismooth at $x$ with respect to the  mapping $\Psi:\R^n\tto\R^{m\times n}$ given by
  \[\Psi(x):=\{-X_{L^*}^T\mv L^*\in \Sp^*F(x,\sigma(x),0)\},\ x\in \dom\Psi:=U,  \]
  provided $F$ is SCD semismooth$^*$ at $(x,\sigma(x),0)$.
\end{theorem}
\begin{proof}
  Due to the definition of the SCD property around a point and Proposition \ref{PropSCDRegOSC} we can find an open neighborhood $U'\times V'\times W'$ of $(\xb,\yb,0)$ such that for every $(x,y,z)\in U'\times V'\times W'$ the mapping $F$ has the SCD property around $(x,y,z)$ and ${\rm scd\,reg\,}F(x,y,z)<\infty$. Next consider the open neighborhood $U'':=\{x\in\Omega\cap U'\mv \sigma(x)\in V'\}$ and choose an open neighborhood $U$ of $\xb$ satisfying $\cl U\subset U''$.  This choice of $U$ ensures that ${\rm scd\,reg\,}F(x,\sigma(x),0)<\infty$, $x\in U$. We now show that $\Psi$ is locally bounded at $x\in U$. Indeed, we have $\gamma:={\rm scd\,reg\,}F(x,\sigma(x),0)<\infty$ and, due to the continuity of $\sigma$ and Proposition \ref{PropSCDRegOSC}, we can find a neighborhood $\tilde U\subset \Omega$ of $x$ satisfying ${\rm scd\,reg\,}F(x',\sigma(x'),0)<\gamma+1$ $\forall x'\in\tilde U$. Hence, $\norm{X^*}<\gamma+1$ for every $X^*\in\Psi(x')$ and our claim is established. In order to show the last assertion, consider $x\in U$ such that $F$ is SCD semismooth$^*$ at $(x,\sigma(x),0)$, set $\gamma:={\rm scd\,reg\,}F(x,\sigma(x),0)$ and choose $\tilde \delta>0$ such that ${\rm scd\,reg\,}F(x',\sigma(x'),0)<\gamma+1$ $\forall x'\in\B_{\tilde \delta}(x)$.  By definition of the SCD semismooth$^*$ property we can find for every $\epsilon>0$ some  $\delta\in (0,\tilde\delta)$ such that for every $x'\in \B_{\delta}(x)$, every subspace $L^*\in \Sp^*F(x',\sigma(x'),0)$ and every $p\in\R^m$ we have
  \begin{align*}\vert \skalp{Z_{L^*}p, 0-0}-\skalp{X_{L^*}p,x'-x}&-\skalp{p,\sigma(x')-\sigma(x)}\vert=\vert \skalp{p,X_{L^*}^T(x'-x)+ \sigma(x')-\sigma(x)}\vert\\
  &\leq \epsilon\norm{(x'-x,\sigma(x')-\sigma(x),0-0)}\norm{(Z_{L^*}p,X_{L^*}p,p)}\\
  &\leq   \epsilon(\norm{x'-x}+\norm{\sigma(x')-\sigma(x)})\sqrt{1+\kappa(L^*)^2}\norm{p}\end{align*}
  and $\kappa(L^*)\leq \gamma+1$. This shows that
  \[\norm{X_{L^*}^T(x'-x)+ \sigma(x')-\sigma(x)}\leq \epsilon(\norm{x'-x}+\norm{\sigma(x')-\sigma(x)})\gamma'\]
  with $\gamma'=\sqrt{1+(1+\gamma)^2}$. Thus $(1-\epsilon \gamma')\norm{\sigma(x')-\sigma(x)}\leq \norm{x'-x}(\norm{X_{L^*}^T}+\epsilon \gamma')$
  implying $\norm{\sigma(x')-\sigma(x)}\leq 2\norm{x'-x}((\gamma+1)+ 1/2)$
  whenever $\epsilon \in (0,1/(2\gamma'))$. Hence
  \[\norm{X_{L^*}^T(x'-x)+ \sigma(x')-\sigma(x)}\leq \epsilon\norm{x'-x}\gamma'(2\gamma +4),\]
  verifying the $\Psi$-ss property of $\sigma$  at $x$. This completes the proof.
\end{proof}

\begin{corollary}\label{CorSCDImp}
  Consider a mapping $F:\R^n\times\R^m\tto\R^m$ and a continuous mapping $\sigma:\Omega\to\R^m$, $\Omega\subset \R^n$ open, satisfying $0\in F(x,\sigma(x))$, $x\in\Omega$. Further assume that for every $x\in\Omega$ the mapping $F$ is SCD semismooth$^*$ at $(x,\sigma(x),0)$ and $\scdreg F(x,\sigma(x),0)<\infty$. Then $\sigma$ is semismooth on $\Omega$ with respect to the mapping $\Psi:\R^n\tto\R^{m\times n}$ given by
  \[\Psi(x):=\{-X_{L^*}^T\mv L^*\in \Sp^*F(x,\sigma(x),0)\},\ x\in \dom\Psi:=\Omega.\]
  In particular, for every  function $\varphi:D\to \R$ defined on an open set $D\subset\R^n\times\R^m$ containing $\gph \sigma$, which is semismooth on $\gph\sigma$ with respect to some mapping $\Phi:\R^n\times\R^m\tto\R^{1\times(n+m)}$, there holds  that the function $\theta:\Omega\to \R$, $\theta(x):=\varphi(x,\sigma(x))$ is semismooth on $\Omega$ with respect to the mapping $\Theta:\R^n\tto\R^{1\times n}$ given by
  \begin{equation*}
  \Theta(x):=
    \{g_x^T -g_y^T X_{L^*}^T\mv (g_x^T,g_y^T)\in\Phi(x,\sigma(x)), L^*\in\Sp^* F(x,\sigma(x),0)\},\ x\in\dom\Theta:=\Omega.
  \end{equation*}
\end{corollary}
Note that both in Corollary \ref{CorSS*Imp} and Corollary \ref{CorSCDImp} the $\Psi$-semismoothness of $\sigma$ on $\Omega$ implies that $\sigma$ is locally Lipschitz and almost everywhere differentiable on $\Omega$.
  However, in Corollary \ref{CorSS*Imp} we require that $\sigma$ is some continuous localization of the solution map $S$, whereas in Corollary \ref{CorSCDImp} the mapping $\sigma$ only needs to be any continuous {\em selection} of $S$. In case of a single-valued mapping $F$, a similar result has already been stated in \cite{Kr18}. However, since most implicit function  theorems (see, e.g., \cite{Gow04,MeSuZh05,PaSuSu03}) guarantee existence  of a
locally unique implicit function, it is clear that the assumptions of such theorems cannot
be satisfied in situations where $\sigma$ is not locally unique.

  To illustrate Corollary \ref{CorSCDImp} consider  the following academic example.
\begin{example}
  Consider the functions $f_1,f_2:\R\times\R\to\R$ and the mapping $F:\R\times\R\tto\R$ given by
  \begin{equation}\label{EqEx_f}f_1(x,y):=\frac 12 y^2-xy,\ f_2:=-\frac 12 y^2,\ F(x,y):=\partial^c_y f(x,y)\ \mbox{with}\  f(x,y):=\max\{f_1(x,y),f_2(x,y)\}.\end{equation}
  Note that the respective solution mapping $S(x)=\{y\mv 0\in F(x,y)\}$ amounts to the {\em C-stationary-point mapping} of the nonsmooth program $\min_y f(x,y)$ parameterized by $x$.
  For every $x\in\R$ we have
  \[f(x,y)=\begin{cases}f_2(x,y)&\mbox{for $y\in[x^-,x^+]$},\\
  f_1(x,y)&\mbox{for $y\not\in(x^-,x^+)$,}\end{cases}\]
  where $x^-:=\min\{x,0\}$ and $x^+:=\max\{x,0\}$. Thus, for any $x\in\R$ we obtain
  \begin{align*}\gph\partial^c_y f(x,\cdot)&= \{(y,y-x)\mv y\not\in[x^-,x^+]\}\cup\{(y,-y)\mv y\in(x^-,x^+)\}\\
  &\qquad\cup(\{x^-\}\times[x^- -x,-x^-])\cup (\{x^+\}\times(-x^+,x^+ -x))\end{align*}
  and we see that $S(x)=\{x,0\}$ does not possess a single-valued localization around $(0,0)$. Let us consider the continuous selection $\sigma(x)=\{x\}=\argmin_yf(x,y)$. We compute now the SC derivative $\Sp F(x,x,0)$ for $x\in\R$. By the above representation of $\gph \partial^c_y f(x,\cdot)$ it follows that $\OO_F$ is the union of the four sets
  \begin{align*}O_1:=\{(x,y,y-x)\mv x\in\R, y\not\in[x^-,x^+]\},\ O_2:=\{(x, y,-y)\mv x\not=0,y\in(x^-,x^+)\}\\
  O_3:=\{(x,x,z^*)\mv x\not=0, z^*\in (x^- -x,x^+ -x)\},\ O_4:=\{(x,0,z^*)\mv x\not=0,z^*\in(-x^+,-x^-)\}.
  \end{align*}
  Moreover, defining the four subspaces
  \begin{align*}L_1:=\{(u,v,v-u)\mv (u,v)\in\R^2\},\ L_2:=\{(u,v,-v)\mv (u,v)\in\R^2\},\\
   L_3:=\{(u,u,v)\mv(u,v)\in\R^2\},\ L_4:=\{(u,0,v)\mv(u,v)\in\R^2\},
  \end{align*}
  we have that $T_{\gph F}(x,y,z^*)=L_k$ for all $(x,y,z^*)\in O_k$ and $k=1,\ldots,4$. Hence, $\Sp F(x,\sigma(x),0)=\{L_1,L_3\}$ for $x\not=0$ and $\Sp F(0,0,0)=\{L_1,L_2,L_3,L_4\}$.
  All the adjoint subspaces
  \[L_1^*=\rge(1,-1,1),\ L_2^*=\rge(-1,0,1),\ L_3^*=\rge(0,-1,1),\ L_4^*=\rge(0,0,1)\]
  belong to $\Z_{1,2}^{\rm reg}$ and thus  we can conclude from Corollary \ref{CorSCDImp} that $\sigma$ is semismooth with respect to
  \begin{equation}\label{EqExamplePsi}\Psi(x):=\begin{cases}\{1\}&\mbox{if $x\not=0$,}\\\{1,0\}&\mbox{if $x=0$.}\end{cases}\end{equation}
  Note that in accordance with Theorem \ref{ThNorkinR_m}(iii) the mapping $\Psi$ is almost everywhere single-valued, although $\Sp F(x,\sigma(x),0)$ does not have this property for any $x$.

These results enable us to apply the considered ImP approach to the following nonsmooth bilevel program
\[\min_{(x,y)\in\R^2\times\R} \varphi(x,y):=2\big\vert y -\vert x_1\vert\big\vert - \vert x_2\vert +\frac 12 x_1^2\quad\mbox{subject to}\quad y\in \argmin f(\eta(x),\cdot),\]
where $f$ is given by \eqref{EqEx_f} and the {\em control} variable $x$ enters now the lower-level problem via the nonsmooth function $\eta:\R\times\R\to\R$ given by $\eta(x)=\vert x_1\vert -\vert x_2\vert$. The solution of this bilevel program is $\xb=(0,0)$, $\yb=0$ and we solved it by applying the BT algorithm to the (reduced)
problem
\[\min \theta(x):=\varphi(x,\sigma(\eta(x))),\]
where $\sigma(\xi):=\argmin_y f(\xi,y)$ is evaluated by solving the nonlinear program
\[\min_{z,y} z\quad\mbox{subject to}\quad z\geq f_1(\xi,y),\ z\geq f_2(\xi,y)\]
by means of the MATLAB routine {\tt fmincon}, which was able to compute in every case the correct global solution.
\if{
Since $\vert \cdot\vert$ and $\sigma$ are semismooth with respect to the mappings
\[t\mapsto\sign(t):=\begin{cases}1&\mbox{if $t\geq0$,}\\-1&\mbox{for $t<0$,}\end{cases}
\quad
\xi\mapsto\tilde\Psi(\xi):=\begin{cases}1&\mbox{if $\xi\not=0$,}\\0&\mbox{if $\xi=0$,}\end{cases}\]
cf. \eqref{EqExamplePsi}, by means of Lemma \ref{LemChainRule} and Corollary \ref{CorSCDImp}
we obtain that $\theta$ is semismooth with respect to the mapping
\[(x_1,x_2)\mapsto\big(2\sign(u(x))(\tilde\Psi(\eta(x))-1)\sign(x_1)+x_1,
-2\sign(u(x))\tilde\Psi(\eta(x))\sign(x_2)-\sign(x_2)\big)
,\]
where $u(x):=\sigma(\eta(x))-\vert x_1\vert$.
}\fi
Since $\vert \cdot\vert$ is G-semismooth, i.e., $\partial\vert\cdot\vert$-ss,  and $\sigma$ is $\Psi$-ss with $\Psi$ given by \eqref{EqExamplePsi}, these functions are also semismooth with respect to the selections
\[t\mapsto\sign(t):=\begin{cases}1&\mbox{if $t\geq0$,}\\-1&\mbox{for $t<0$,}\end{cases}
\quad
\xi\mapsto\tilde\Psi(\xi):=\begin{cases}1&\mbox{if $\xi\not=0$,}\\0&\mbox{if $\xi=0$}\end{cases}\]
of $\partial\vert\cdot\vert$ and $\Psi$, respectiveley.
By means of Lemma \ref{LemChainRule} and Corollary \ref{CorSCDImp}
we obtain that $\theta$ is semismooth with respect to the mapping
\[(x_1,x_2)\mapsto\big(2\sign(u(x))(\tilde\Psi(\eta(x))-1)\sign(x_1)+x_1,
-2\sign(u(x))\tilde\Psi(\eta(x))\sign(x_2)-\sign(x_2)\big)
,\]
where $u(x):=\sigma(\eta(x))-\vert x_1\vert$.
We did not make any attempt to compute exact Clarke subgradients of $\theta$ because, according to our theory, elements of a suitable semismooth derivative may replace them very well. Indeed, starting from $\ee x0=(5,-1)$, the BT algorithm stopped after $16$ Iterations and $21$ evaluations of the oracle at the point $(-2.8\cdot 10^{-6},1.8\cdot 10^{-8})$. Other starting points yield similar results.
\end{example}
\section{Conclusion}
The paper contains a deep analysis of some generalizations of the classical semi\-smoothness property. Namely, it deals with the so-called $\Psi$-semismoothness and new notions of $\Gamma$-semismoothness$^*$ and SCD-semismoothness$^*$. These notions are closely related to each other and have a number of useful properties concerning, in particular, some fundamental numerical approaches to nonsmooth problems. For instance, whereas the standard rules of  generalized differential calculus attain typically the form of inclusions, the calculus of semismooth derivatives yields again a semismooth derivative. This is especially helpful in connection with a family of bundle methods in nonsmooth optimization.

A special attention has been paid to the semismoothness of localizations and selections of solution maps to parameter-dependent inclusions and their possible compositions with real-valued objectives. The presented analysis of these mappings paves the way to an efficient extension of the ImP approach to a broader class of problems with equilibrium constraints.

\section*{Acknowledgment} The authors would like to thank both reviewers for their numerous helpful comments improving the quality of the paper.

\end{document}